\documentclass[12pt]{amsart}
\usepackage{amsmath,enumitem}
\usepackage[english,  activeacute]{babel}
\usepackage[latin1]{inputenc}
\usepackage{amssymb}
\usepackage{amsthm}
\usepackage{graphics}
\usepackage{array}
\usepackage{pdflscape}
\usepackage{a4wide}
\usepackage{color, url}
\usepackage{float}

\theoremstyle{plain}
\newtheorem{theorem}{Theorem}

\newtheorem{lemma}[theorem]{Lemma}
\newtheorem{proposition}[theorem]{Proposition}

\theoremstyle{definition}
\newtheorem{definition}[theorem]{Definition}
\newtheorem{example}[theorem]{Example}
\newtheorem{remark}[theorem]{Remark}

\newcommand{\E}{{\mathbb E}}
\newcommand{\N}{{\mathbb N}}
\newcommand{\M}{{\mathbb M}}
\newcommand{\La}{{\mathbb L}}
\newcommand{\Pa}{{\mathcal P}}
\newcommand{\Block}{{\mathfrak B}}
\newcommand{\Ber}{{\mathbb B}}
\newcommand{\Odd}{{\mathbb O}}
\newcommand{\Z}{{\mathbb Z}}

\newcommand{\T}{{\mathbb T}}
\newcommand{\U}{{\mathbb U}}
\newcommand{\I}{{\mathbb I}}

\def\li{\text{\rm Li}}
\def\lif{\text{\rm Lif}}
\title[Restricted $r$-Stirling Numbers and their Combinatorial Applications]{Restricted $r$-Stirling Numbers and their Combinatorial Applications}

\author{Be\'ata B\'enyi}
\address{\noindent Faculty of Water Sciences, National University of Public Service, Budapest, HUNGARY}
\email{beata.benyi@gmail.com}

\author{Miguel M\'{e}ndez}
\address{\noindent Departamento de Matem\'aticas, Facultad de Ciencias, Universidad Antonio Nari\~{n}o, Bogot\'a,  COLOMBIA}
\email{mmendezenator@gmail.com}

\author{Jos\'e L. Ram\'{\i}rez}
\address{\noindent Departamento de Matem\'aticas, Universidad Nacional de Colombia, Bogot\'a,  COLOMBIA}
\email{jlramirezr@unal.edu.co}

\author{Tanay Wakhare}
\address{\noindent University of Maryland, College Park, MD 20742, USA}
\email{twakhare@gmail.com}

\date{\today}
\subjclass[2010]{Primary 11B83, 11B73 ; Secondary 05A19, 05A15.}
\keywords{Set partitions, generalized Stirling numbers,  generating functions, combinatorial identities.}

\begin{document}
\begin{abstract}
We study set partitions with $r$ distinguished elements and block sizes found in an arbitrary index set $S$. The enumeration of these $(S,r)$-partitions leads to the introduction of $(S,r)$-Stirling numbers, an extremely wide-ranging generalization of the classical Stirling numbers and the $r$-Stirling numbers. We also introduce the associated $(S,r)$-Bell and $(S,r)$-factorial numbers. We study fundamental aspects of these numbers, including recurrence relations and determinantal expressions. For $S$ with some extra structure, we show that the inverse of the $(S,r)$-Stirling matrix encodes the M\"obius functions of two families of posets. Through several examples, we demonstrate that for some $S$ the matrices and their inverses involve the enumeration sequences of several combinatorial objects. Further, we highlight how the $(S,r)$-Stirling numbers naturally arise in the enumeration of cliques and acyclic orientations of special graphs, underlining their ubiquity and importance. Finally, we introduce related $(S,r)$ generalizations of the poly-Bernoulli and poly-Cauchy numbers, uniting many past works on generalized combinatorial sequences.

\end{abstract}

\maketitle

\section{Introduction}
Set partitions of a finite set are an important and classical topic in enumerative combinatorics. Extensive work has been conducted concerning enumeration of the total number of set partitions under certain constraints (cf. \cite{mansour}). We define a \emph{partition of a set} $[n]:= \{1, 2, \dots , n\}$ as a collection of pairwise disjoint subsets, called \emph{blocks}, whose union is $[n]$. For a block $\Block$, we denote the cardinality of the block $\Block$ by $|\Block|$. The sequence counting the total number of set partitions of $[n]$ into $k$ non-empty blocks is the \emph{Stirling numbers of the second kind}, denoted by ${n \brace k}$.

There are several important generalizations of the Stirling numbers. One of them is the \emph{$r$-Stirling numbers of the second kind} introduced by Broder \cite{Broder}.   Letting $r$ be a non-negative integer, the $r$-Stirling numbers of the second kind, ${n \brace k}_r$, are defined as the number of set partitions of $[n + r]$ into $k + r$ blocks with the additional condition that the first $r$ elements are in distinct blocks.  The partitions where the first $r$ elements are in distinct blocks are  called \emph{$r$-partitions}, and the elements $1, 2, \dots , r$ are called \emph{special elements}.    It is clear that if $r=0$  we obtain the Stirling numbers of the second kind.

For example, ${2 \brace 1}_{2}=5$, with the relevant partitions being
\begin{gather*}
\{\{\overline{1}\}, \, \{\overline{2}\}, \, \{3, 4\}\}, \quad  \{\{\overline{1}, 3\}, \, \{\overline{2}\}, \, \{4\} \}, \quad \{\{\overline{1}, 4\}, \, \{\overline{2}\}, \, \{3\} \}, \\
\{\{\overline{1}\}, \, \{\overline{2}, 3\}, \, \{4\}\}, \quad \{\{\overline{1}\}, \, \{\overline{2}, 4\}, \, \{3\} \}.
\end{gather*}
Notice that the special elements are overlined.

The $r$-Stirling numbers of the second kind satisfy the following recurrence \cite{Broder}:
\begin{align}\label{defrStir}
{n \brace k}_r=&(k+r){n-1 \brace k}_r+{n-1\brace k-1}_r, \quad n\geq k,
\end{align}
with  ${n \brace k}_r=0$ if $n<k$ and ${n \brace k}_r=1$ if $n=k$.

Mez\H{o} \cite{Mezo} defined the \emph{$r$-Bell numbers}, $B_{n, r}$,  as the number of $r$-partitions of an $n+r$-element set. This is equivalent to
$$B_{n,r}=\sum_{m=0}^n{n \brace m}_r.$$

A natural generalization of the $r$-Stirling number of the second kind arises from considering the restriction that all block sizes are  contained in a set $S\subseteq \Z^{+}$. For $n, k, r \geq 0$ and $S\subseteq \Z^+$, we let $\Pi_{S,r}(n,k)$ denote the set of all $r$-partitions of $[n+r]$ into $k+r$ non-empty blocks, such that the cardinality of each block is contained in the set $S$.  We call this kind of partition an \emph{$(S,r)$-partition}. In particular, we let ${n\brace k}_{S,r}$ denote the cardinality of the set $\Pi_{S,r}(n,k)$, and call this sequence the \emph{$(S,r)$-Stirling numbers of the second kind}.   The total number of $(S,r)$-partitions of $[n+r]$ is the \emph{$(S,r)$-Bell number} $B_{n,S,r}$. It is clear that
$$B_{n,S,r}=\sum_{k=0}^n{n\brace k}_{S,r}.$$

Recently,  Mihoubi and Rahmani \cite{MR2017} studied this new sequence  as a generalization of the partial Bell polynomials. If $S=\{k_1, k_2, \dots \}$, then we have the following  exponential generating functions:
\begin{align}
\sum_{n=k}^\infty { n \brace k }_{S,r} \frac{x^n}{n!}=\frac{1}{k!}\left(\sum_{i\geq 1} \frac{x^{k_i-1}}{(k_i-1)!}\right)^r\left(\sum_{i\geq 1} \frac{x^{k_i}}{k_i!}\right)^k, \label{def:rrs2}\\
\sum_{n=0}^\infty B_{n,S,r} \frac{x^n}{n!}=\left(\sum_{i\geq 1} \frac{x^{k_i-1}}{(k_i-1)!}\right)^r\exp\left(\sum_{i\geq 1} \frac{x^{k_i}}{k_i!}\right).
\end{align}

It is clear that we recover the $r$-Stirling numbers by setting $S=\Z^+=\{1,2,3,\ldots\}$. If we take $S=\{1, 2,  \dots, m\}$ we obtain the restricted $r$-Stirling numbers of the second kind \cite{KJL2}. In a similar way, if we take $S=\{m, m+1, \dots\}$, we recover the associated $r$-Stirling  numbers of the second kind \cite{KJL2}.  Moreover, if $r=0$ we have the $S$-restricted Stirling numbers of the second kind \cite{BenRam, Eng, Wakhare}. \\

We can also obtain the following general trivariate generating function by formally summing \eqref{def:rrs2} over $r$ and $k$, and interchanging the order of summation:
\begin{equation}
\sum_{r=0}^\infty \sum_{n=0}^\infty \left(\sum_{k=0}^n  { n \brace k }_{S,r}y^k \right) \frac{z^rx^n}{r!n!} = \exp\left(y\sum_{i\geq 1} \frac{x^{k_i}}{k_i!} \right)\exp\left(z\sum_{i\geq 1} \frac{x^{k_i-1}}{(k_i-1)!} \right).
\end{equation}
The paranthesized summand can be regarded as an $(S,r)$ generalization of a Bell polynomial, and will be considered later.

In this paper we study $(S,r)$-partitions, building on the work of Mihoubi and Rahmani. In particular, we prove several new combinatorial identities, and provide combinatorial proofs for some known identities. Using the theory of Riordan matrices we present determinantal identities for the generalized Bell and factorial sequences. {Moreover, for $S$ with a specific structure, we give combinatorial interpretations for the inverses of the $(S,r)$-Stirling matrices of both kinds}. Additionally, we present some examples of  $(S,r)$-partitions which naturally arise in graph theory. Finally, we introduce a new family of polynomials which generalizes the poly-Bernoulli numbers and poly-Cauchy numbers.

\section{Some Combinatorial Properties}
\subsection{Recurrence relations}
First, we derive some fundamental recurrence relations satisfied by the $(S,r)$-Stirling numbers and associated $(S,r)$-Bell numbers. We mainly provide combinatorial proofs, but all of the following results have generating function proofs. Consider the generating function for ${n\brace k}_{S,r}$ in the form

\begin{equation}\label{Sgenfunc}
\sum_{n=k}^\infty { n \brace k }_{S,r} \frac{x^n}{n!}=\frac{1}{k!}\left(\sum_{s\in S} \frac{x^{s-1}}{(s-1)!}\right)^r\left(\sum_{s \in S} \frac{x^{s}}{s!}\right)^k.
\end{equation}
A generalization of \ref{Sgenfunc}, the partial $r$-Bell polynomials, was recently introduced by Mihoubi and Rahmani \cite{MR2017}. To convert between their notation and ours we note that their $B^{(r)}_{n+r,k+r}(a_\ell,b_\ell)$ corresponds to our ${n\brace k}_{S,r}$, with
$$a_\ell=b_\ell = \begin{cases}
1,& \thinspace \ell\in S, \\
0,&\thinspace 0\in S;
\end{cases}$$
i.e.,  $a_\ell=b_\ell$ is the indicator function for whether $\ell$ is in our index set $S$. Therefore we can directly use some of their results, while adding some new ones of our own. The following theorem follows from appropriately specializing Mihoubi and Rahmanis' results. We provide combinatorial proofs of independent interest, in contrast to their generating function based proofs.

\begin{theorem}\label{teo1}\cite[Prop. 3]{MR2017}
We have the following recurrences:
\begin{align}
k{n\brace k}_{S,r} &=\sum_{s\in S}  \binom{n}{s}{n-s\brace k-1}_{S,r} \label{id1},  \\
r{n\brace k}_{S,r} &=\sum_{s\in S}r \binom{n}{s-1} {n-s+1\brace k}_{S,{r-1}} \label{id2},
\end{align}
and
\begin{align}
(n+r){n\brace k}_{S,r} =  \sum_{s\in S}s \binom{n}{s} {n-s\brace k-1}_{S,{r}} + r \sum_{s\in S}s \binom{n}{s-1} {n-s+1\brace k}_{S,{r-1}}.  \label{id3}
\end{align}
\end{theorem}
\begin{proof}
The left-hand side of \eqref{id1} counts the total number of elements in $\Pi_{S,r}(n,k)$ such that one of the non-special blocks is coloured. Suppose that the coloured non-special block has size $s$, so that this block can be constructed in $\binom{n}{s}$ ways. The remaining $n-s$ non-special elements form a $(S,r)$-partition into  $k-1$ non-empty blocks, which can be constructed in ${n-s\brace k-1}_{S,r}$ ways. Summing over $s$ completes the argument.\\

For the second identity \eqref{id2}, the left-hand side counts $(S,r)$-partitions with a coloured special block (which is equivalent to saying ``with a coloured special element"). Consider the case where the coloured block has size $s$. Such a partition can be obtained by first choosing the coloured special element, then choosing $s-1$ non-special elements for this coloured block (in $\binom{n}{s-1}$ ways), and constructing from the remaining $n-(s-1)+(r-1)$ elements a $(S,r-1)$-partition (in ${n-s+1\brace k}_{S,{r-1}}$ ways). Summing over $s$ completes the argument.\\

Finally, the left-hand side of \eqref{id3} counts the $(S,r)$-partitions with a single coloured element (special or non-special). The coloured element is in a special or a non-special block. First, assume that it is in a non-special block of size $s$, so that it must be a non-special element. There are $s\binom{n}{s}$ ways to choose the $s$ elements for the block and mark one of the elements in the block. The remaining $n-s+r$ elements are partitioned into $k-1+r$ blocks in ${n-s\brace k-1}_{S,r}$ ways. Assume now that the coloured element is in a special block of size $s$. Choose a special element (in one of $r$ ways) and $s-1$ non-special elements in one of $\binom{n}{s-1}$ ways for the block; now, mark one of the elements of the block, the special or non-special element, in one of $s$ ways, and construct a $(S,r-1)$-partition of the remaining $n-(s-1)+(r-1)$ elements into $(k+r-1)$ non-empty blocks (in ${n-s+1\brace k}_{S,r-1}$ ways). Summing over $s$ completes the argument.\\

Note that it is possible to give an algebraic proof of these identities. For example, for \eqref{id2} we begin with the generating function \eqref{Sgenfunc} and write it as a product of two power series. First, we downshift the summation index from $n=k$ to $n=0$, since ${n\brace k}_{S,r} = 0$ for $n<k$. Therefore,
\begin{align*}
\sum_{n=0}^\infty { n \brace k }_{S,r} \frac{x^n}{n!} &=\frac{1}{k!}\left(\sum_{s\in S} \frac{x^{s-1}}{(s-1)!}\right)^{r-1}\left(\sum_{s \in S} \frac{x^{s}}{s!}\right)^k  \left(\sum_{s\in S} \frac{x^{s-1}}{(s-1)!}\right)  \\
&= \left(\sum_{n=0}^\infty { n \brace k }_{S,r-1} \frac{x^n}{n!} \right)\left(\sum_{\substack{n=0\\n+1\in S}}^\infty \frac{x^{n}}{n!}\right)  \\
&= \sum_{n=0}^\infty x^n \sum_{\substack{j=0\\j+1\in S}}^n \frac{1}{j!(n-j)!} {n-j\brace k}_{S,r-1}.
\end{align*}
Reindexing the summation to go over $S$ and comparing coefficients of $x^n$ completes the proof.
\end{proof}

We can obtain a slightly more complicated recurrence as follows:
\begin{theorem}\label{teo8a}
We have the recurrence
$$
{n+1\brace k}_{S,r}  = {n\brace k-1}_{S,r+1} + r\sum_{s\in S}  \binom{n}{s-2}{n-s+2\brace k}_{S,r-1} .
$$
\end{theorem}
\begin{proof}
We begin with the generating function \eqref{Sgenfunc} beginning at $n=0$, take a derivative, and compare coefficients. Therefore,
\begin{align*}
\sum_{n=0}^\infty { n+1 \brace k }_{S,r} \frac{x^n}{n!} &=\frac{1}{k!}\left(\sum_{s\in S} \frac{x^{s-1}}{(s-1)!}\right)^r\left(\sum_{s \in S} \frac{x^{s}}{s!}\right)^{k-1}  k\left(\sum_{s\in S} \frac{x^{s-1}}{(s-1)!}\right)  \\
&\medspace+ \frac{1}{k!}\left(\sum_{s\in S} \frac{x^{s-1}}{(s-1)!}\right)^{r-1}\left(\sum_{s \in S} \frac{x^{s}}{s!}\right)^{k}  r\left(\sum_{s\in S-\{1\}} \frac{x^{s-2}}{(s-2)!}\right) \\
&=  \sum_{n=0}^\infty { n \brace k-1 }_{S,r+1} \frac{x^n}{n!} +  r\sum_{n=0}^\infty { n \brace k }_{S,r-1} \frac{x^n}{n!}  \sum_{s\in S-\{1\}} \frac{x^{s-2}}{(s-2)!}.
\end{align*}
Reindexing the summation to go over $S$ and comparing coefficients of $x^n$ completes the proof.\\

We also provide a combinatorial argument. The left-hand side counts the $(S,r)$-partitions of $[n+1+r]$ into $k+r$ blocks. Consider the position of the $(n+1)$-th element. It is contained in a non-special or a special block. Suppose first that it is contained in a non-special block. Considering $(n+1)$ as a special element, we actually have a $(S,r+1)-$partition with $k-1$ non-special blocks (since the block containing $(n+1)$ is now a special block). The number of such partitions is counted by ${n \brace k-1}_{S,r+1}$. Suppose now that the $(n+1)$-th element is contained in a special block of size $s$. This block contains a special element $r$ and $s-2$ other elements, hence it can be constructed in $r\binom{n}{s-2}$ ways. From the remaining $(n-s+2)+(r-1)$ elements we can construct a $(S,r-1)$-partition into $k$ non-special blocks in ${n-s+2\brace k}_{S,r-1}$ ways.
\end{proof}

Though Mihoubi and Rahman did not consider analogs of the Bell numbers, we can easily use the previous identities to describe similar recurrences for $(S,r)$-Bell numbers. In general, given a recurrence for ${n\brace k}_{S,r}$ that trivially depends on $k$, we can formally sum over all $k$ from $0$ to $\infty$ to obtain a Bell number identity. This follows from the fact $B_{n,S,r} = \sum_{k=0}^\infty {n\brace k}_{S,r}$, with the $k>n$ terms equal to $0$. Therefore, we easily obtain the following set of identities:
\begin{theorem}
We have the following recurrences:
\begin{align*}
B_{n,S,r+1} &=\sum_{s\in S}  \binom{n}{s-1}B_{n-s+1,S,r},\\
(n+r)B_{n,S,r} &=  \sum_{s\in S}s \binom{n}{s} B_{n-s,S,{r}} + r \sum_{s\in S}s \binom{n}{s-1} B_{n-s+1,S,{r-1}},\\
B_{n+1,S,r}  &=B_{n,S,r+1} + r\sum_{s\in S}  \binom{n}{s-2}B_{n-s+2,S,r-1}.
\end{align*}
\end{theorem}

\subsection{Changing the index set}
Some of the most interesting results about these numbers occur when we shift the set $S$. For the rest of this section, let $S+\vec{a} = \{s+a|s\in S\}$, where $a \in \Z$ can also be negative.

The following result is due to Mihoubi and Rahman.
\begin{theorem}\label{teo9a} \cite[Prop. 1]{MR2017}
If $1\in S$, we have the recurrence
$$
{n\brace k}_{S,r} =\sum_{i=0}^r \sum_{j=0}^k \binom{r}{i}\binom{n}{j}{n-j\brace k-j}_{S- \{1\},r-i}.
$$
\end{theorem}
\begin{proof}
The left-hand side counts the $r$-partitions of $[n+r]$ into $k+r$ blocks. For the right-hand side we count these $r$-partitions according to the number of singletons. Suppose that there are $i$ special blocks of size 1 and    $j$ non-special blocks of size 1. Then there are  $\binom{r}{i}\binom{n}{j}$ ways to construct these blocks. The remaining $(n-j)+(r -i)$ elements must be  arranged in non-singleton  blocks, that is  in ${n-j\brace k-j}_{S- \{1\},r-i}$ ways.  Summing over $i$ and $j$ completes the argument.
\end{proof}
From the recurrence above we obtain the following relation for the $(S,r)$-Bell numbers:
$$B_{n,S,r}=\sum_{i=0}^r\sum_{j=0}^n\binom ri \binom nj B_{n-j,S-\{1\}, r-i}.$$

In the following theorem we generalize the combinatorial identity given in Theorem \ref{teo9a}.

\begin{theorem}\label{teo9aa}
If $u \in S$, we have the recurrence
$$
{n\brace k}_{S,r} =\sum_{i=0}^r \sum_{j=0}^k \binom{r}{i}\frac{n!}{(u-1)!^i u!^j j! (n-(u-1)i-uj)!}{n-(u-1)i-uj \brace k-j}_{S- \{u\},r-i}.
$$
\end{theorem}
\begin{proof}
To show this identity, we count $r$-partitions according to their number of blocks of size $u$. Suppose that there are $i$ special blocks of size exactly $u$ and $j$ non-special blocks of size $u$. We first construct the blocks of size $u$. For these blocks we need $i$ special elements from the $r$ distinguished elements and $(u-1)i+uj$ non-special elements from the remaining $n$. We choose these elements in $\binom{r}{i}$ ways, resp. in $\binom{n}{(u-1)i+uj}$ ways. We construct the blocks from the chosen elements in $\frac{((u-1)i+uj)!}{(u-1)!^ii!u!^jj!}$ ways. There are $i!$ ways to insert our $i$ special elements into the special blocks.
Hence, we have
$$
\binom{r}{i}\binom{n}{(u-1)i+uj}\frac{((u-1)i+uj)!}{(u-1)!^ii!u!^jj!}i!=\frac{n!}{(u-1)!^iu!^jj!(n-(u-1)i-uj)!}
$$
possibilities for constructing the blocks of size $u$.
The remaining $(n-(u-1)i-uj)+(r-i)$ elements must be arranged in $(k+r)-(i+j)$ blocks with the restriction that the size of the blocks are contained in the set $S-\{u\}$ and that the remaining $r-i$ special elements are in distinct blocks. Hence, we have ${n-(u-1)i-uj \brace k-j}_{S-\{u\},r-i}$ possible constructions.  Summing over $i$ and $j$ completes the argument.
\end{proof}

We can now obtain a reduction formula for $r$ which also reduces the set $S$.
\begin{theorem}
Let $\ell \in \Z$ with $0\leq \ell \leq r$. Then we have the recurrence
$$
{n\brace k}_{S,r} = \ell!\sum_{j=0}^n \binom{n}{j} {j\brace k}_{S,r-\ell} {n-j\brace \ell}_{S-\vec{1}}=\sum_{j=0}^n \binom{n}{j} {j\brace k}_{S,r-\ell} {n-j\brace 0}_{S,\ell}.
$$
\end{theorem}
\begin{proof}
We begin, as always, with the generating function \eqref{Sgenfunc}. Then
\begin{align*}
\sum_{n=0}^\infty { n \brace k }_{S,r} \frac{x^n}{n!}&=\frac{1}{k!}\left(\sum_{s\in S} \frac{x^{s-1}}{(s-1)!}\right)^r\left(\sum_{s \in S} \frac{x^{s}}{s!}\right)^k \\
&=\frac{1}{k!}\left(\sum_{s\in S} \frac{x^{s-1}}{(s-1)!}\right)^{r-\ell}\left(\sum_{s \in S} \frac{x^{s}}{s!}\right)^k \left(\sum_{s\in S} \frac{x^{s-1}}{(s-1)!}\right)^{\ell}\\
&=\sum_{n=0}^\infty { n \brace k }_{S,r-\ell} \frac{x^n}{n!}  \ell ! \sum_{n=0}^\infty { n \brace \ell }_{S-\vec{1}} \frac{x^n}{n!}.
\end{align*}
Taking a product and comparing coefficients completes the proof of the first identity. The proof of the second identity is similar.
\\
We also provide a combinatorial proof. Let $j$ be the number of coloured non-special elements. We count the $(S,r)$-partitions according to the special blocks that do not contain any coloured elements. Let $\ell$ denote the number of such special blocks. Choose first the $j$  elements that will be coloured in $\binom{n}{j}$ ways. There are ${j\brace k}_{S,r-\ell}$ ways to construct $(k+r-\ell)$ blocks, such that a special block contains coloured elements, and ${n-j\brace 0}_{S,\ell}$ ways to construct the $\ell$ special blocks without coloured non-special elements. This proves the second identity. We can also construct the $\ell$ special blocks without coloured non-special elements such that we partition the $n-j$ elements into $\ell$ blocks such that each block has size in $S-\vec{1}$ in ${n-j \brace \ell}_{S-\vec{1}}$ ways, and add one of the $\ell$ special element to each block in $\ell!$ ways. This proves the first identity.
\end{proof}

This result also suggests a more in-depth combinatorial study of ${n\brace 0}_{S,r}$ as a special limit case.
\begin{proposition}
If $1 \notin S$, then
$${n\brace 0}_{S,r}=r!{n\brace r}_{S-\vec{1}}.$$
If $1 \in S$, then
$$
{n\brace 0}_{S,r}=\sum_{i=0}^r(r)_i{n\brace i}_{S-\{1\}},
$$
where $(r)_i:=r(r-1)(r-2)\cdots(r-i+1)$ is a falling factorial.
\end{proposition}
\begin{proof}
Note that ${n\brace 0}_{S,r}$ counts partitions of $n+r$ into $r$ blocks, each block containing a special element. Deleting the special elements we obtain an $(S-\vec{1},r)$-partition into $r$ non-empty blocks. Otherwise, there are $r!$ ways to augment each block with an element of $[r]$.

Assume now that $S$ contains $1$.  In this case there are some, say $i$, singleton blocks containing only a special element. Choose the special elements for the singleton blocks in ${r \choose i}$ ways, and apply the same argument as before for the remaining construction to obtain $(r-i)!{n\brace r-i}_{S-\{1\}}$.
\end{proof}

\section{$(S,r)$-restricted Permutations}
The goal of this section is to study an analogous restriction for the case of permutations. The \emph{$(S,r)$-Stirling numbers of the first kind},  denoted by   ${n \brack k}_{S,r}$, enumerate the number of permutations of a set with $n+r$ elements into $k+r$ cycles such that the first $r$ elements are in different cycles and all cycle sizes  are  contained in the set $S\subseteq \Z^{+}$.  The permutations where the first $r$ elements are in distinct cycles are  called \emph{$r$-permutations}.  The elements $1, 2, \dots , r$ will also be called \emph{special elements}, and the cycles with special elements will be called \emph{special cycles}.

The exponential generating function of the sequence  ${n\brack k}_{S,r}$ is
\begin{equation}\label{Sgenfunc2}
\sum_{n=k}^\infty { n \brack k }_{S,r} \frac{x^n}{n!}=\frac{1}{k!}\left(\sum_{s\in S} x^{s-1}\right)^r\left(\sum_{s \in S} \frac{x^{s}}{s}\right)^k.
\end{equation}
We recover the $r$-Stirling numbers of the first kind by setting $S=\Z^+$. If we take $S=\{1, 2,  \dots, m\}$ we obtain the restricted $r$-Stirling numbers of the first kind \cite{KJL2}. Similarly, if we take $S=\{m, m+1, \dots\}$, we recover the associated $r$-Stirling  numbers of the first kind \cite{KJL2}.

The $(S,r)$-Stirling numbers of the first kind  ${n\brack k}_{S,r}$ coincide with the partial $r$-Bell polynomials  $B^{(r)}_{n+r,k+r}(a_\ell,b_\ell)$ \cite{MR2017} with
$$a_\ell=b_\ell = \begin{cases}
(\ell-1)!,& \thinspace \ell\in S, \\
0,&\thinspace 0\in S;
\end{cases}$$

Therefore we can use Proposition 3 in \cite{MR2017} to deduce some simple recurrences, for which we will provide combinatorial proofs.

\begin{theorem}\cite[Prop. 3]{MR2017}
We have the following recurrences:
\begin{align}
k{n\brack k}_{S,r} &=\sum_{s\in S} (s-1)!\binom{n}{s}{n-s\brack k-1}_{S,r} \label{id1s1}  \\
r{n\brack k}_{S,r} &=\sum_{s\in S} r(s-1)!\binom{n}{s-1} {n-s+1\brack k}_{S,{r-1}} \label{id2s1}
\end{align}
and
\begin{align}
(n+r){n\brack k}_{S,r} =  \sum_{s\in S}s! \binom{n}{s} {n-s\brack k-1}_{S,{r}} + r \sum_{s\in S}s! \binom{n}{s-1} {n-s+1\brack k}_{S,{r-1}}.  \label{id3s1}
\end{align}
\end{theorem}
\begin{proof}
The left-hand side of \eqref{id1s1} counts the total number of $r$-permutations  with a coloured non-special cycle.  If the coloured non-special cycle has size $s$, then there are $(s-1)!\binom{n}{s}$ ways to construct this cycle, and  the  remaining $n-s+r$ elements are arranged in ${n-s\brack k-1}_{S,r}$ ways. Summing over $s$ completes the argument.

The proofs of the remaining identities follow  a similar argument to that used in Theorem \ref{teo1}.
\end{proof}

The proofs of the following theorems  follow those of Theorems   \ref{teo8a} and \ref{teo9a}.
\begin{theorem}\label{teo8b}
We have the recurrence
$$
{n+1\brack k}_{S,r}  = {n\brack k-1}_{S,r+1} + r\sum_{s\in S} (s-1)! \binom{n}{s-2}{n-s+2\brack k}_{S,r-1} .
$$
\end{theorem}

\begin{theorem}\label{teo9b}\cite[Prop. 1]{MR2017}
If $1\in S$, we have the recurrence
$$
{n\brack k}_{S,r} =\sum_{i=0}^r \sum_{j=0}^k \binom{r}{i}\binom{n}{j}{n-j\brack k-j}_{S- \{1\},r-i}.
$$
\end{theorem}

\begin{theorem}
If $u \in S$, we have the recurrence
$$
{n\brack k}_{S,r} =\sum_{i=0}^r \sum_{j=0}^k \binom{r}{i}\frac{n!}{u^j j! (n-(u-1)i-uj)!}{n-(u-1)i-uj \brack k-j}_{S- \{u\},r-i}.
$$
\end{theorem}

\begin{proof}
The proof of the theorem follows the proof of Theorem \ref{teo9aa}. Assume that there are $j$ non-special and $i$ special cycles of length exactly $u$ in the permutation.
First we choose the necessary $i$ special element in $\binom{r}{i}$ ways and $i(u-1)+ju$ non-special elements in $\binom{n}{(u-1)i+uj}$, which exhausts the cycles of length $u$. For a permutation of $[i(u-1)+ju]$ , we associate these $j+i$ cycles
as follows: take $u$ elements as a cycle, $j$ times, then take $u-1$ elements $i$ times. For each $u-1$ elements insert one of the chosen special elements as a starting element. The insertion of the special elements can be done in $i!$ ways. However, this double counts some permutations; if we permute the non-special cycles, as well as the order of the special cycles we obtain the same associated permutation. Furthermore, for a non-special cycle we could choose any $u$ element to start the cycle. Hence, we have
$$
\binom{r}{i}\binom{n}{(u-1)i+uj}\frac{((u-1)i+uj)!}{u^jj!}=\binom{r}{i}\frac{n!}{u^j j! (n-(u-1)i-uj)!}
$$
total ways to construct the relevant cycles. The remaining $n-i(u-1)-ju$ elements form a $(S-\{u\},r-i)$-permutation.
\end{proof}

\section{$(S,r)$-restricted Stirling Matrices}

As a next step we use the algebraic theory of Pascal and Stirling matrices, and the theory of Riordan groups \cite{Riordan} respectively, for the study of our sequences. We introduce

the \emph{$(S,r)$-Stirling matrix of the second kind} and  \emph{$(S,r)$-Stirling matrix of the first kind}, as the infinite matrices defined by
$$
\M_{S,r}:=\left[{n \brace k}_{S,r}\right]_{n, k \geq 0} \quad \mbox{and}\quad \La_{S,r}:=\left[{n \brack k}_{S,r}\right]_{n, k \geq 0} .
$$
An infinite lower triangular matrix $L=\left[d_{n,k}\right]_{n,k\in \N}$ is called an \emph{exponential  Riordan array}, (cf. \cite{Barry}), if its column $k$ has generating function $g(x)\left(f(x)\right)^k/k!, k = 0, 1, 2, \dots$, where $g(x)$ and $f(x)$ are formal power series with $g(0) \neq 0$, $f(0)=0$ and $f'(0)\neq 0$.   The matrix corresponding to the pair $f(x), g(x)$ is denoted by  $\langle g(x),f(x)\rangle$.

If we multiply  $\langle g(x),f(x)\rangle$ by a column vector $(c_0, c_1, \dots)^T$ with exponential generating function $h(x)$, then the resulting column vector has exponential generating function $g(x)h(f(x))$. This property is known as the fundamental theorem of exponential Riordan arrays.  The product of two exponential Riordan arrays $\langle g(x),f(x)\rangle$ and $\langle h(x),\ell(x)\rangle$ is then defined by:
 $$\langle g(x),f(x)\rangle *\langle h(x),\ell(x)\rangle =\left\langle g(x)h\left(f(x)\right), \ell\left(f(x)\right)\right\rangle.$$
 The set of all exponential Riordan matrices is a group  under the operator $*$ (cf.\ \cite{Barry, Riordan}).\\

  For example,  the Pascal matrix $\Pa$, the Stirling matrix of the second kind  $\mathcal{S}_2$,  and the Stirling matrix of the first kind  $\mathcal{S}_1$  are all given by the Riordan matrices:

 \begin{align*}
 \Pa=\langle e^x,x\rangle=&\left[\binom nk \right]_{n, k \geq 0}, \quad \quad   \mathcal{S}_2=\langle 1,e^x-1\rangle=\left[{n \brace k} \right]_{n, k \geq 0},  \quad \quad   \\
&\mathcal{S}_1=\langle 1, -\ln(1-x) \rangle=\left[{n \brack k} \right]_{n, k \geq 0}.
 \end{align*}

From Equations \eqref{Sgenfunc} and  \eqref{Sgenfunc2}, and the definition of Riordan matrix we obtain the following theorem.
\begin{theorem}
For all $S\subseteq \Z^+$ with $1\in S$, the matrices $\M_{S,r}$ and $\La_{S,r}$ are exponential Riordan matrices given by
$$\M_S=\left\langle \left(\sum_{s\in S} \frac{x^{s-1}}{(s-1)!}\right)^r, \sum_{s \in S} \frac{x^{s}}{s!}  \right\rangle \quad \quad \La_S=\left\langle \left(\sum_{s\in S} x^{s-1} \right)^r,\sum_{s \in S} \frac{x^{s}}{s} \right\rangle.$$
\end{theorem}
It is clear that the row sum of the matrix  $\M_{S,r}$ are the $(S,r)$-Bell numbers $B_{n, S, r}$.

The inverse exponential Riordan array of $\M_{S,r}$ and $\La_{S,r}$ are denoted by
$$\T_{S,r}:=\left[{n \brace k}_{S,r}^{-1}\right]_{n, k \geq 0} \quad  \text{and} \quad \U_{S,r}:=\left[{n \brack k}_{S,r}^{-1}\right]_{n, k \geq 0}.$$

For the particular case $r=0$,  Engbers et al. \cite{Eng} gave an interesting combinatorial interpretation for the absolute  values of the entries ${n \brace k}_S^{-1}$ and ${n \brack k}_S^{-1}$ by using Schr\"oder trees.

Since $\M_{S,r}*\T_{S,r} = \I,$ where $\I$ is the identity matrix, we have the orthogonality relation:
\begin{align*}
\sum_{i=k}^n {n \brace i}_{S,r}  {i \brace k}_{S,r}^{-1}&=\sum_{i=k}^n  {n \brace i}_{S,r}^{-1} {i \brace k}_{S,r}
=\delta_{k,n}.
\end{align*}

The orthogonality relation gives us the inverse relation:
\begin{align*}
f_n=\sum_{k=0}^n{n \brace k}_{S,r}^{-1} g_k  \iff g_n&=\sum_{k=0}^n{n \brace k}_{S,r} f_k.
\end{align*}
Let us introduce the $(S,r)$-Bell polynomials by
$$B_{n, S, r}(x):=\sum_{k=0}^n{n \brace k}_{S,r}x^k.$$
From  the definition of the polynomials $B_{n, S, r}(x)$ we obtain the  equality:
$$X = \M_{S,r}^{-1}\mathcal{B}_{S,r},$$
where $X=[1, x, x^2, \dots]^T$ and $\mathcal{B}_S=[B_{0,S,r}(x), B_{1,S,r}(x), B_{2,S,r}(x), \dots]^T$. Further,  $X=\T_{S,r}\mathcal{B}_{S,r}$ and
$$x^n=\sum_{k=0}^n {n \brace k}_{S,r}^{-1} B_{k,S,r}(x).$$
 Therefore,
 \begin{align}
 B_{n,S,r}(x)=x^n-\sum_{k=0}^{n-1}{n \brace k}_{S,r}^{-1} B_{k,S,r}(x), \quad n\geq 0.\label{iddet}
 \end{align}
 From the above identity we obtain a determinantal identity for $B_{n,S,r}(x)$.
 \begin{theorem}\label{teo3}
 For all $S\subseteq \Z^+$ with $1\in S$, the $(S,r)$-Bell polynomials satisfy
 $$B_{n,S,r}(x)=(-1)^n\begin{vmatrix}
 1 & x & & \cdots && x^{n-1} & x^n\\
  1 & {1 \brace 0}_{S,r}^{-1} & &\cdots && {n-1 \brace 0}_{S,r}^{-1} & {n \brace 0}_{S,r}^{-1}\\
    0 & 1 && \cdots && {n-1 \brace 1}_{S,r}^{-1}  & {n \brace 1}_{S,r}^{-1} \\
   \vdots &  & & \cdots& & & \vdots\\
  0 & 0  && \cdots & & 1 & {n \brace n-1}_{S,r}^{-1}\\
 \end{vmatrix}.$$
 \end{theorem}
 \begin{proof}
This identity follows from Equation \eqref{iddet} and by expanding the determinant by the last column.
 \end{proof}

 For example, if $S=\{1, 3, 8\}$ and $r=2$, then
\begin{align*}
\M_{\{1, 3, 8 \}, 2}&=\left\langle \left(1 + \frac{x^2}{2!} + \frac{x^7}{7!}\right)^2 , x + \frac{x^3}{3!} + \frac{x^8}{8!} \right\rangle\\
&=\left(
\begin{array}{ccccccccc}
 1 & 0 & 0 & 0 & 0 & 0 & 0 & 0 & 0 \\
 0 & 1 & 0 & 0 & 0 & 0 & 0 & 0 & 0 \\
 0 & 0 & 1 & 0 & 0 & 0 & 0 & 0 & 0 \\
 0 & 7 & 0 & 1 & 0 & 0 & 0 & 0 & 0 \\
 0 & 0 & 16 & 0 & 1 & 0 & 0 & 0 & 0 \\
 0 & 50 & 0 & 30 & 0 & 1 & 0 & 0 & 0 \\
 0 & 0 & 220 & 0 & 50 & 0 & 1 & 0 & 0 \\
 0 & 210 & 0 & 700 & 0 & 77 & 0 & 1 & 0 \\
 0 & 17 & 2240 & 0 & 1820 & 0 & 112 & 0 & 1 \\
   \vdots  &  &  &  & \vdots &  &  &  & \vdots \\
\end{array}
\right),
\end{align*}
and
$$\T_{\{1, 3, 8\},2}=\left[{n \brace k}_{\{1, 3, 8\},2}^{-1}\right]_{n, k \geq 0}=
\left(
\begin{array}{ccccccccc}
 1 & 0 & 0 & 0 & 0 & 0 & 0 & 0 & 0 \\
 0 & 1 & 0 & 0 & 0 & 0 & 0 & 0 & 0 \\
 0 & 0 & 1 & 0 & 0 & 0 & 0 & 0 & 0 \\
 0 & -7 & 0 & 1 & 0 & 0 & 0 & 0 & 0 \\
 0 & 0 & -16 & 0 & 1 & 0 & 0 & 0 & 0 \\
 0 & 160 & 0 & -30 & 0 & 1 & 0 & 0 & 0 \\
 0 & 0 & 580 & 0 & -50 & 0 & 1 & 0 & 0 \\
 0 & -7630 & 0 & 1610 & 0 & -77 & 0 & 1 & 0 \\
 0 & -17 & -38080 & 0 & 3780 & 0 & -112 & 0 & 1 \\
\end{array}
\right).$$
Notice that in this example the sequence $a(n)=7,16,30,50,77,112,\ldots$ (A00581 in \cite{OEIS}) arises in both matrices (up to sign). One of the combinatorial interpretations of these numbers is the following: let $X$ be a $[n+2]$ element set and $Y$ a $2$-subset of $X$, then $a(n)_{n\geq 1}$ is the number of $(n-1)$-subsets of $X$ intersecting $Y$.

The first few $(\{1, 3, 8\},2)$-Bell polynomials are
\begin{align*}
&1, \quad x, \quad x^2, \quad  x^3+7 x, \quad x^4+16 x^2, \quad x^5+30 x^3+50 x, \quad x^6+50 x^4+220 x^2,\\
&x^7+77x^5+700 x^3+210 x, \quad x^8+112 x^6+1820 x^4+2240 x^2+17 x, \dots
\end{align*}
   In particular,
   \begin{align*}
   B_{6, \{1, 3, 8\}, 2}(x)=x^6+50 x^4+220 x^2 =\begin{vmatrix}
 1 & x & x^2 & x^3 & x^4 & x^{5} & x^6\\
 1 & 0 & 0 & 0 & 0 & 0 & 0\\
    0 & 1 & 0& -7 & 0& 160  & 0 \\
    0 & 0 & 1& 0 & -16& 0  & 580 \\
    0 & 0 & 0&1 & 0& -30  & 0 \\
     0 & 0 & 0&0 & 1& 0  & -50 \\
     0 & 0 & 0&0 & 0& 1  & 0  \end{vmatrix}.
 \end{align*}
Analogously to the definition of $(S,r)$-Bell polynomials, we can define  the
\emph{$(S,r)$-factorial polynomials} $A_{n,S,r}(x)$ by the expression
$$A_{n,S,r}(x)=\sum_{k=0}^n{n\brack k}_{S,r}x^k.$$
Notice that if $S=\Z^+$ and $r=0$, then $A_{n, \Z^+,0}(1)=n!$. Some of their combinatorial and arithmetical properties for the cases $S=\{1, 2, \dots, m\}$ and $S=\{m, m+1, \dots\}$ have been studied in \cite{Moll}.

From a similar argument, we have the following theorem:
 \begin{theorem}\label{teo3bb}
 For all $S\subseteq \Z^+$ with $1\in S$, the $(S,r)$-factorial polynomials satisfy
 $$A_{n,S,r}(x)=(-1)^n\begin{vmatrix}
 1 & x & & \cdots && x^{n-1} & x^n\\
  1 & {1 \brack 0}_{S,r}^{-1} & &\cdots && {n-1 \brack 0}_{S,r}^{-1} & {n \brack 0}_{S,r}^{-1}\\
    0 & 1 && \cdots && {n-1 \brack 1}_{S,r}^{-1}  & {n \brack 1}_{S,r}^{-1} \\
   \vdots &  & & \cdots& & & \vdots\\
  0 & 0  && \cdots & & 1 & {n \brack n-1}_{S,r}^{-1}\\
 \end{vmatrix}.$$
 \end{theorem}

\section{Combinatorial interpretations of the $(S,r)$-Stirling matrices and their inverses}
In this section we provide a combinatorial interpretation of the inverses of the $(S,r)$-Stirling matrices of the first and second kind. For this purpose, we introduce posets whose Möbius function is given by these matrices.

\subsection{Composition-partition pairs and ordered composition-permutation pairs}
An $r$-composition of a set $\mathbf{U}$ is an $r$-tuple of disjoint sets $\mathbf{U}=(U_1,U_2,\dots,U_r)$ (some of which may be empty)  whose union is $U$,
$$U_1\uplus U_2\uplus\dots\uplus U_r=U.$$
We consider pairs of the form $(\mathbf{V}, \pi)$ where $\mathbf{V}$ is an $r$-composition of a subset $V$ of $U$ and $\pi$ is a set partition of the complementary set $U-V$. Such objects will be called {\em composition-partition pairs}.

We have the following interpretation of the $(S,r)$-Stirling numbers of the second kind.  For shorter notation, let $S'$ denote the set of integers that we obtain by reducing each integer in $S$ by $1$, $S'=\{s-1|s\in S\}$. We call $S'$ the {\em derivative} of $S$.
\begin{proposition}\label{stirl2}
	The $(S,r)$-Stirling number of the second kind ${n\brace k}_{S,r}$ counts the composition-partition pairs $(\mathbf{V}, \pi)$ over a set of $n$ elements satisfying the following two conditions.
	\begin{enumerate}
		\item \label{cond1} The sizes of the sets in the composition are all in $S'$.
		\item \label{cond2} The partition $\pi$ has exactly $k$ blocks, the sizes of each of which is in $S$.
		\end{enumerate}
\end{proposition}
\begin{proof}
	We establish a bijection between the set $\Pi_{S,r}(n,k)$ and the given composition-partition pairs. Let $\Block_1,\Block_2,\dots, \Block_r, \Block_{r+1},\dots, \Block_{k+r}$ be the blocks of a partition in $\Pi_{S,r}(n,k)$ arranged in such a way that the first $r$ blocks ($\Block_1,\Block_2,\dots, \Block_r$) contain the elements $1,\ldots, r$; i.e., $i\in \Block_i$ for $i=1,2,\dots,r.$ Define the composition-partition pair $(\mathbf{V},\pi)$ as follows: $V_i=\Block_i-\{i\}$ and $\pi=\{\Block_{r+1},\Block_{r+2},\dots,\Block_{r+k}\}.$ The composition-partition pair $(\mathbf{V},\pi)$ (over the $n$-element set $\{r+1,r+2,\dots,r+n\}$) clearly satisfies the conditions. The correspondence is obviously reversible and hence bijective.
\end{proof}
We use the notation $\Pi_{S,r}(n,k)$ for the set of composition-partition pairs $(\mathbf{V},\pi) $ described in Proposition \ref{stirl2}, and $\Pi_{S,r}(n)$ for the same kind of composition-partition pairs without restrictions on the number of blocks of $\pi$.

A similar combinatorial interpretation can be given for the $(S,r)$-Stirling numbers of the first kind.
For our purposes it is useful to fix a a certain order of the elements in the cycles and of the disjoint cycles. In our notation each cycle lists its least element first and the cycles are sorted in increasing order by their first element. For instance,  $(1\,5\,7)(2\,4\,6)(3\,8)$. Fixing this convention, we view a cycle as a linear order of its elements.
We consider compositions enriched with linear orders,
$\boldsymbol{\ell}=(\ell_1,\ell_2,\dots,\ell_r)$, each $\ell_i$ being a linear order on the set $V_i$, with $\mathbf{V}=(V_1,V_2,\dots,V_r)$ being an $r$-composition of some set $V$. Such a tuple $\boldsymbol{\ell}$ will be called an {\em ordered composition}.

\begin{proposition}\label{prop.fstirling}
	The $(S,r)$-Stirling number of the first kind ${n\brack k}_{S,r}$ count the ordered composition-permutation pairs $(\boldsymbol{\ell}, \sigma)$ over a set of $n$ elements, satisfying the following two conditions.
	\begin{enumerate}
		\item  The sizes of the linear orders in the composition are all in $S'$.
		\item The permutation $\sigma$ has exactly $k$ cycles, the sizes of each of them being in $S$.
	\end{enumerate}
\end{proposition}
The proof of the Proposition \ref{prop.fstirling} is analogue to the proof of Proposition \ref{stirl2}.
We denote by $\mathbf{P}_{S,r}(n,k)$ the set of all ordered compositions-permutation pairs and by  $\mathbf{P}_{S,r}(n)$ all these pairs without restrictions on the number of cycles in $\sigma$.

\subsection{The special struture of $S$}
In order to give a combinatorial interpretation of the inverses of restricted Stirling matrices of the first and second kind, we have to assume that the set $S$ has a special structure. This structure is necessary to construct a partial order on the composition-partition (resp. ordered composition-permutation) pairs, whose M\"obius functions will give us the respective inverse matrices $\T_{S,r}$ and $\U_{S,r}$.

\begin{definition}\label{1+monoid}Let $S$ be a subset of $\Z^+$. It is said to be a $^+1$-monoid if
	\begin{enumerate}\item $1\in S$; and
		\item for every sequence $s_1, s_2,\dots, s_\ell$ of elements in $S$ with $\ell\in S$, we have that
 the sum $\sum_{j=1}^\ell s_j$ is also in $S$.
		\end{enumerate}	
\end{definition}
As a consequence of the definition of $^+1$-monoid we obtain the following proposition.
\begin{proposition}\label{moduleS}
		If $S$ is a $^+1$-monoid then for every sequence $s_1, s_2,\dots, s_\ell$ of elements in $S$, where the integer $\ell$ is  in $S'$, the sum $\sum_{j=1}^\ell s_j$ is in $S'$.
\end{proposition}
\begin{proof}
	Setting $s_{\ell+1}=1$, all the elements of the sequence $s_1, s_2,\dots, s_\ell, s_{\ell+1}$ are in $S$, and $\ell+1$ is also in $S$. Then  $s_1+s_2+\dots+s_\ell+s_{\ell+1}=s_1+s_2+\dots+s_\ell+1$ is in $S$, and the result follows.
\end{proof}

\begin{example}The set of odd positive integers $\Odd^+=\{2k+1|k=0,1,2,\dots\}$ is a $^+1$-monoid. If we add an odd number of odd integers, the result is again an odd integer.  The same argument can be applied for the set of positive integers congruent to one modulo a fixed positive integer, $S_q=\{qk+1|k=0,1,2,\dots\}$.
	
Observe that the set $S_q'$ of multiples of $q$ is a submonoid of $\N$.
\end{example}
	
A simpler and equivalent condition for a set $S$ to be a $^+1$-monoid is the following.
\begin{lemma}\label{lem6}
	A set $S$ is a $^+1$-monoid if and only if
	\begin{enumerate}
		\item $1\in S$; and
		\item if $s_1, s_2\in S$, then $s_1+s_2-1\in S$.
	\end{enumerate}
 \end{lemma}
\begin{proof}Assume that $S$ is a $^+1$-monoid. Letting $k_1=s_1, k_2=1, k_3=1,\dots, k_{s_2}=1,$ we obtain, since $s_2\in S$, that $k_1+k_2+\dots+k_{s_2}=s_1+s_2-1\in S$. Conversely, assume that  $s_1, s_2,\dots, s_\ell$ and $\ell$ are in $S$. Iteratively applying condition $2$ of the lemma, we have
\begin{align*}
s_1+\ell-1\in S& \Rightarrow s_1+s_2+\ell-2\in S\\
&\Rightarrow s_1+s_2+s_3+\ell-3\in S\\
&\ \ \vdots\\
& \Rightarrow s_1+s_2+\dots+s_\ell\in S.
\end{align*}
\end{proof}

As a direct consequence of the Lemma \ref{lem6}, we get the following proposition.
\begin{proposition}\label{monoid}
	The derivative of a $^+1$-monoid is a submonoid of the natural numbers under addition. Conversely, if a set $S'$ is a submonoid of $\N$, then $S=\{s+1|s\in S'\}$ is a $^+1$-monoid.
\end{proposition}
\begin{proof}The proof is easy, and left to the reader.\end{proof}

\subsection{The partial order on the sets  $\Pi_{S,r}(n)$ and $\mathbf{P}_{S,r}(n)$}
Now we are able to define a partial order on the set $\Pi_{S,r}(n)$. Its M\"obius function is encoded by the matrix $\T_{S,r}$.

We introduce two operations on the set $\Pi_{S,r}(n)$. Let $(\mathbf{V},\pi) \in \Pi_{S,r}(n)$. We obtain another composition-partition pair $(\mathbf{V}',\pi')\in \Pi_{S,r}(n)$ by the following two operations.
\begin{enumerate}
\item\label{op1} The compositions remain unchanged, i.e., $\mathbf{V}=\mathbf{V}'$, and  $\pi'$ is obtained from $\pi$ by joining $\ell$ blocks of $\pi$, for some $\ell\in S$.
\item\label{op2} The components of $\mathbf{V}'$ are the same as the components of $\mathbf{V}$ except one, say $V_j$, which is augmented by some blocks of $\pi$; while $\pi$ is reduced by these blocks. Precisely:
\begin{eqnarray*}
		V'_i&=&V_i,\mbox{ for $i\neq j$, and } V'_j=V_j\cup\bigcup_{i=1}^\ell \Block_i,\; \ell\in S'; \mbox{and}\\
		\pi'&=&\pi-\{\Block_1, \Block_2,\dots, \Block_\ell\}.
		\end{eqnarray*}
\end{enumerate}
From now on, we follow the convention of separating the composition-partition ordered pair $(\mathbf{V},\pi )$ with double bars  $$\mathbf{V}||\pi :=(\mathbf{V},\pi ).$$
\begin{example}
Consider the set $\Pi_{\Odd,2}(6)$, where $\Odd$ denotes the set of odd integers, and its element $(\{1,2\},\emptyset)||3|4|5|6$. By operation $\ref{op1}$ we obtain for instance the element $(\{1,2\},\emptyset)||3\,4\,5|6$ and  by operation \ref{op2} the element $(\{1,2\},\{4,5\})||3|6$.
\end{example}
The operations (\ref{op1}) and (\ref{op2}) are closed on $\Pi_{S,r}(n)$; operation  (\ref{op1}) by Definition \ref{1+monoid}, operation (\ref{op2}) by Propositions \ref{moduleS} and \ref{monoid}. Hence we can define the following partial order on the set $\Pi_{S,r}(n)$.
\begin{definition}\label{def.orderst1}
Let $(\mathbf{V},\pi)$ and $(\mathbf{V}',\pi')$ be two elements of $\Pi_{S,r}(n)$. We will say that $(\mathbf{V},\pi)\leq (\mathbf{V}',\pi')$ if $(\mathbf{V}',\pi')$ is obtained from $(\mathbf{V},\pi)$ by any combination of the two above operations (\ref{op1}) and (\ref{op2}).
\end{definition}

The poset $\Pi_{S,r}(n)$ has a least element $\widehat{0}=(\emptyset,\emptyset,\dots,\emptyset)||1|2|\dots|n$. When $n\in S$, the maximal elements are of the form $\mathbf{V}||\emptyset$, $\mathbf{V}$ being a composition of $[n]$.

\begin{example}\label{stirling2inverse}
	For $S$ the set of positive odd integers, the Riordan matrix $\M_{S,r}$ is  given by
	$$ \M_{S,r}=\left\langle (\cosh(x))^r,\sinh(x) \right\rangle.$$ The inverse matrix $\T_{S,r}$ is given by
	\begin{equation}\label{inversesin}\T_{S,r}= \left\langle \cosh^{-r}(\sinh^{\langle -1\rangle}(x)),\sinh^{\langle -1\rangle}(x) \right\rangle=\left\langle (1+x^2)^{-r/2},\sinh^{\langle -1\rangle}(x) \right\rangle.\end{equation}
	For $r=2$, the first few rows and columns of $\M_{S,2}$ are
	$$ \M_{S,2}=\left\langle (\cosh(x))^2,\sinh(x) \right\rangle=	
	\left(
	\begin{array}{ccccccccc}
	1 & 0 & 0 & 0 & 0 & 0 & 0 & 0 & 0 \\
	0 & 1 & 0 & 0 & 0 & 0 & 0 & 0 & 0 \\
	2 & 0 & 1 & 0 & 0 & 0 & 0 & 0 & 0 \\
	0 & 7 & 0 & 1 & 0 & 0 & 0 & 0 & 0 \\
	8 & 0 & 16 & 0 & 1 & 0 & 0 & 0 & 0 \\
	0 & 61 & 0 & 30 & 0 & 1 & 0 & 0 & 0 \\
	32 & 0 & 256 & 0 & 50 & 0 & 1 & 0 & 0 \\
	0 & 547 & 0 & 791 & 0 & 77 & 0 & 1 & 0 \\
	128 & 0 & 4096 & 0 & 2016 & 0 & 112 & 0 & 1 \\
	\end{array}
	\right).$$
	We will see (Theorem \ref{inversecombina} below) that the $n$th row of the inverse matrix
	\begin{align*}
	\T_{S,2}&=\left\langle (1+x^2)^{-1},\sinh^{\langle -1\rangle}(x) \right\rangle\\
	&=\left(
	\begin{array}{ccccccccc}
	1 & 0 & 0 & 0 & 0 & 0 & 0 & 0 & 0 \\
	0 & 1 & 0 & 0 & 0 & 0 & 0 & 0 & 0 \\
	-2 & 0 & 1 & 0 & 0 & 0 & 0 & 0 & 0 \\
	0 & -7 & 0 & 1 & 0 & 0 & 0 & 0 & 0 \\
	24 & 0 & -16 & 0 & 1 & 0 & 0 & 0 & 0 \\
	0 & 149 & 0 & -30 & 0 & 1 & 0 & 0 & 0 \\
	-720 & 0 & 544 & 0 & -50 & 0 & 1 & 0 & 0 \\
	0 & -6483 & 0 & 1519 & 0 & -77 & 0 & 1 & 0 \\
	40320 & 0 & -32768 & 0 & 3584 & 0 & -112 & 0 & 1 \\
	\end{array}
	\right)
\end{align*}
	 encodes the M\"obius function of the poset $\Pi_{S,2}(n)$. We have that
	$$\T_{S,2}(n,k)=\sum_{(\mathbf{V},\pi)\in \Pi_{S,2}(n,k)}\mu\, (\widehat{0},(\mathbf{V},\pi)).$$
	
	 For example, its $4$th row is $(24,0,-16,0,1)$. Since $\widehat{0}=(\emptyset,\emptyset)||1|2|3|4$, we have
	 $$1=\T_{S,2}(4,4)=\mu(\widehat{0},\widehat{0}),\quad  -16=\T_{S,2}(4,2)=\sum_{(\mathbf{V},\pi)\in \Pi_{S,2}(4,2)}\mu\, (\widehat{0},(\mathbf{V},\pi)),$$ and $$ 24=\T_{S,2}(4,0)=\sum_{(\mathbf{V},\emptyset)\in \Pi_{S,2}(4,0)}\mu\, (\widehat{0},(\mathbf{V},\pi)).$$
	 We  manually check this. The elements of $\Pi_{S,2}(4,2)$ are of two kinds:
 \begin{enumerate}
 	\item\label{k1} Those where $\pi$ has two singletons. They are of the form $(\{a_1,a_2\},\emptyset)||a_3|a_4$, or $(\emptyset,\{a_1,a_2\})||a_3|a_4.$
 	\item\label{k2} Those where $\pi$ has a block of size $3$ and a singleton block; i.e., elements of the form $(\emptyset,\emptyset)||a_1\, a_2\, a_3|a_4.$
 \end{enumerate}
There are $2\times\binom{4}{2}=12$ elements of type (\ref{k1}) and $4$ elements of type (\ref{k2}). Each of them covers $\widehat{0}$, the M\"obius function of each of them is equal to $-1$, and their sum equals $-16$.

The elements of  $\Pi_{S,2}(4,0)$ are also of two kinds:
\begin{enumerate}
	\item\label{p1} Those where each component of the composition has two elements. They have the form $(\{a_1,a_2\},\{a_3,a_4\})||\emptyset$.
	\item\label{p2} Those where one of the components has the whole set $\{1,2,3,4\}$ and the other the emptyset. There are only two elements of this kind.
\end{enumerate}
There are $6=\binom{4}{2}$ elements of type (\ref{p1}). Each of them covers exactly two elements, since $(\{a_1,a_2\},\{a_3,a_4\})||\emptyset$ covers $(\{a_1,a_2\},\emptyset)||a_3|a_4$ and $(\emptyset,\{a_3,a_4\})||a_1|a_2$. Hence, each has M\"obius function equal to $1$. Their contribution to the sum is then $6$. The element $(\{1,2,3,4\},\emptyset)||\emptyset$ covers all the elements of the form $(\{a_1,a_2\},\emptyset)||a_3|a_4$, and all of the form $(\emptyset,\emptyset)||a_1, a_2, a_3|a_4$. The number of all of them is $6+4=10$. Then, its M\"obius function is equal to $9$. Then the contribution of the elements of type (\ref{p2}) is $18$. The sum of the M\"obius function on elements of both types is equal to $24$, as expected.
\end{example}

Now we turn our attention to the ordered case, to an analogue definition of a partial order on the set  $\mathbf{P}_{S,r}(n)$.
We denote the concatenation of linear orders (cycles) by the $+$ symbol. The result of concatenation of two cycles is also a cycle. For example, $(1\,4\,3)+(8\,10\,9)+(2\,5\,7\,6)=(1\,4\,3\,8\,10\,9\,2\,5\,7\,6)$.

We define two operations on the set $\mathbf{P}_{S,r}(n)$. Let $(\boldsymbol{\ell},\sigma)\in \mathbf{P}_{S,r}(n)$ be given. The element $(\boldsymbol{\ell}',\sigma')$ is obtained the following ways.
\begin{enumerate}
\item\label{operationt1} The ordered compositions remain unchanged $\boldsymbol{\ell}=\boldsymbol{\ell}'$, and $\sigma'$ is obtained from $\sigma$ by concatenating $s$ cycles of $\sigma$, for some $s\in S$ in any order.
	\item\label{operationt2} The components of the ordered compositions remain unchanged, except one, say $\ell'_j$; which is obtained by concatenation of the corresponding component $\ell_j$  with $s$ cycles of $\sigma$ (in any order), $s$ being an element of $S'$. More formally,
		\begin{eqnarray*}
			\ell'_i&=&\ell_i,\mbox{ for $i\neq j$, and } \ell'_j=\ell_j+c_{j_1}+c_{j_2}+\dots+c_{j_s},\; s\in S'; \mbox{and}\\
			\sigma'&=&\sigma-\{(c_{j_1}),(c_{j_2}),\dots, (c_{j_s})\}.
		\end{eqnarray*}
\end{enumerate}

\begin{example}
Let $S=\Odd$  be the set of odd integers, and $(5\,3,7\,9)||(1\,4\,6)(2\,8\,10)(11)$ an element of the set $\mathbb{P}_{\Odd,2}(11)$. By applying operation \ref{operationt1} we obtain the pair $(5\,3,7\,9)||(1\,4\,6\,11\,2\,8\,10)$, since $(\textcolor{blue}{1\,4\,6}\,\textcolor{green}{11}\,\textcolor{red}{2\,8\,10})=(1\,4\,6)+(11)+(2\,8\,10)$ is a sum of $3$ cycles ($3\in S$). On the other hand, by operation \ref{operationt2} we obtain the pair $(5\,3,7\,9\,\,11\,1\,4\,6)||(2\,8\,10)$, since  $\textcolor{blue}{7\,9}\,\,\textcolor{green}{11}\,\textcolor{red}{1\,4\,6}=7\,9+11+1\,4\,6$ is a sum of a linear order and two cycles ($2\in S'$).

\end{example}
\begin{definition}\label{def.orderstir2}
Let $(\boldsymbol{\ell},\sigma)$ and $(\boldsymbol{\ell}',\sigma')$ be two elements of $\mathbf{P}_{S,r}(n)$. We say that $(\boldsymbol{\ell},\sigma)\leq (\boldsymbol{\ell}',\sigma')$ if $(\boldsymbol{\ell}',\sigma')$ is obtained from $(\boldsymbol{\ell},\sigma)$ by any combination of the two above operations (\ref{operationt1}) and (\ref{operationt2}).
\end{definition}

The poset $\mathbf{P}_{S,r}(n)$ has a least element $\widehat{0}=(\emptyset,\emptyset,\dots,\emptyset)||(1)(2)\dots (n)$. The maximal elements are of the form $\boldsymbol{\ell}||\emptyset$, $\boldsymbol{\ell}$ being an ordered composition over $[n]$ and $\emptyset$ being the empty permutation.

The posets $\Pi_{S,r}(n)$, $\mathbf{P}_{S,r}(n)$ can be defined equivalently as follows. The equivalence with Definitions \ref{def.orderst1} and \ref{def.orderstir2} is easy to verify.

\begin{definition} \label{def.orders1}
Let $(\mathbf{V},\pi)$  and $(\mathbf{V}',\pi')$ be two elements of $\Pi_{S,r}(n)$. We have that $(\mathbf{V},\pi)\leq (\mathbf{V}',\pi')$ if
\begin{enumerate}
	\item every component of $\mathbf{V}'$ is obtained as the union of the corresponding component of $\mathbf{V}$ with $t$ blocks of $\pi$, $t\in S'$; and
	\item every block of $\pi'$ is obtained as a union of $s$ blocks of $\pi$, $s\in S$.
\end{enumerate}
\end{definition}
\begin{definition}\label{def.orders2}
Let $(\boldsymbol{\ell},\sigma)$ and $(\boldsymbol{\ell}',\sigma')$ be two elements of $\mathbf{P}_{S,r}(n)$.
	\noindent We have that $(\boldsymbol{\ell},\sigma)\leq(\boldsymbol{\ell}',\sigma')$ if
	\begin{enumerate}
	\item every component of $\boldsymbol{\ell}'$ is obtained as the concatenation of the corresponding component of $\boldsymbol{\ell}$ with $t$ cycles of $\sigma$ (in any order), $t\in S'$; and
	 \item every cycle in $\sigma'$ is the concatenation of $s$ cycles of $\sigma$, $s\in S$.
	\end{enumerate}
\end{definition}

\subsection{Combinatorial Interpretation}
For the proof of our main theorem of this section we will need the following lemma.
\begin{lemma}\label{coideal}Let $S$ be a $^+ 1$-monoid. Let us consider
	 $(\mathbf{V},\pi)$ and $(\boldsymbol{\ell},\sigma)$ elements of $\Pi_{S,r}(n)$ and $\mathbf{P}_{S,r}(n)$,  respectively. Let $k\leq n$ be the number of blocks of $\pi$, and assume that the number of cycles of $\sigma$ is also equal to $k$. Then, we have
	 \begin{eqnarray}\label{eq.coideal1}
	 |\{(\mathbf{V}', \pi'):(\mathbf{V}',\pi')\geq (\mathbf{V},\pi),\; |\pi'|=j\}|&=&|\Pi_{S,r}(k,j)|,\\\label{eq.coideal2}|\{(\boldsymbol{\ell}', \sigma'):(\boldsymbol{\ell}',\sigma')\geq (\boldsymbol{\ell},\sigma),\; |\sigma'|=j\}|&=&|\mathbf{P}_{S,r}(k,j)|.
	 \end{eqnarray}\end{lemma}
	 \begin{proof}For each $(\mathbf{V}',\pi')\geq (\mathbf{V},\pi)$  with $j=|\pi'|$, we are going to construct a unique element $(\mathbf{W},\kappa)\in\Pi_{S,r}(k,j)$. First we order the elements of $\pi$, $\pi=\{\Block_1,\Block_2,\dots, \Block_k\}$.
	 	By part $(1)$ of Definition \ref{def.orders1},  each $V_i'$   is of the form
	 	\begin{equation}\label{eq.recover1}V_i'=V_i\cup\bigcup_{h\in W_i}\Block_h\end{equation}
	 	for some subset $W_i$ of $[k]$
	 	satisfying $|W_i|\in S'$ ($W_i$ might be empty, $0\in S'$).
By part (2) of Definition \ref{def.orders1}, for each block $\Block$ of $\pi'$, there exists a subset $K_B$  of $[k]$, such that
	 \begin{equation}\label{eq.recover2}\Block=\bigcup_{j\in K_B} \Block_j, \end{equation}
	\noindent where $|K_B|\in S$.
	 Let $\mathbf{W}=(W_1,W_2,\dots,W_r)$ and $\kappa=\{K_B|\Block \in\pi'\}$.	Define the correspondence
	 	$$(\mathbf{V}',\pi')\stackrel{\phi}{\mapsto} (\mathbf{W},\kappa).$$
	 	Clearly $|\kappa|=j$, and hence $(\mathbf{W},\kappa)\in \Pi_{S,r}(k,j)$. It is easy to check that $\phi$ is a bijection.  Given $(\mathbf{V},\pi)$ and $(\mathbf{W},\kappa)$, we recover $(\mathbf{V}',\pi')$ by Equations (\ref{eq.recover1}) and (\ref{eq.recover2}).
\begin{example}
For $r=2$ and $S$ the set of odd integers we have that
	 	\begin{multline*}
		(V_1,V_2)||\{\Block_1,\Block_2,\Block_3,\Block_4,\Block_5,\Block_6,\Block_7,\Block_8,\Block_9\}\\
		\leq (V_1\cup \Block_1\cup \Block_3,V_2\cup \Block_2\cup \Block_5)||\{\Block_4 \cup \Block_6\cup \Block_7, \Block_8, \Block_9\}
		\end{multline*}
	 	whatever the blocks of the partitions (of odd size) and the elements of the  composition (of even size) are. The bijection $\phi$ acts as follows
	 	$$(V_1\cup \Block_1\cup \Block_3,V_2\cup \Block_2\cup \Block_5)||\{\Block_4 \cup \Block_6\cup \Block_7, \Block_8, \Block_9\}\stackrel{\phi}{\mapsto} (\{1,3\},\{2,5\})||4 6 7|8|9.$$
\end{example}
	Let $\sigma=(c_1)(c_2)\dots (c_k)$, the cycles ordered in such way that the minimum element of $c_i$ is less than the minimum of $c_{i+1}$, $i=1,2,\dots,k-1$. Assume that  $(\boldsymbol{\ell}',\sigma')\geq (\boldsymbol{\ell},\sigma)$ and the number of cycles of $\sigma'$ is $j$. By Definition \ref{def.orders2} part $(1)$, for every $i=1,2,\dots,r$, there exists a linear order $\widehat{\ell}_i$ of a subset of $[k]$ (that might be empty), $|\widehat{\ell}_i|\in S'$,  such that
	 \begin{equation}\label{eq.reverse1}\ell_i'=\ell_i+\sum_{h\in\widehat{\ell}_i}(c_h).\end{equation}  The concatenation in the sum is made following the order of $\widehat{\ell}_i$.
	  By Definition \ref{def.orders2}, part (2), for every cycle $c\in \sigma'$ there exists a cycle $\gamma_{(c)}$ on some subset of $[k]$, $|\,\gamma_{(c)}\,|\in S$, such that
	 \begin{equation}\label{eq.reverse2}(c)=\sum_{h\in \gamma_{(c)}}(c_h).\end{equation}
	 The concatenation in the sum is made following the order in $\gamma_{(c)}$, so the least cycle in $\{(c_h)|h\in\gamma_{(c)}\}$ goes first.  That guarantees that $(c)$ is a cycle.
	 	Let $\widehat{\boldsymbol{\ell}}=(\widehat{\ell}_1,\widehat{\ell}_2,\dots,\widehat{\ell}_r)$ and $\widehat{\sigma}$ be the permutation whose cycles are of the form $\gamma_{(c)}$, $(c)\in \sigma'.$ It is clear that $(\widehat{\boldsymbol{\ell}},\widehat{\sigma})$ is in $\mathbf{P}_{S,r}(k,j)$. The correspondence
	 	$$(\boldsymbol{\ell}',\sigma')\stackrel{\psi}{\mapsto} (\widehat{\boldsymbol{\ell}},\widehat{\sigma})$$
	 	is the desired bijection. Given $(\boldsymbol{\ell},\sigma)$ and $(\widehat{\boldsymbol{\ell}},\widehat{\sigma})$  we can recover $(\boldsymbol{\ell}',\sigma')$ by Equations (\ref{eq.reverse1}) and (\ref{eq.reverse2}).
	 \end{proof}
\begin{example} Let $S$ be the set of odd integers and $r=2$, we have that $$(\ell_1,\ell_2)||(c_1)(c_2)(c_3)(c_4)(c_5)(c_6)(c_7)(c_8)(c_9)\leq
(\ell_1+c_2c_1c_9c_3,\ell_2+c_5c_4)||(c_6 c_7c_8).$$
whatever the cycles (of odd size) and the linear orders (of even size) are. The bijection $\psi$ acts as follows
$$(\ell_1+c_2c_1c_9c_3,\ell_2+c_5c_4)||(c_6 c_7 c_8)\stackrel{\psi}{\mapsto} (2\, 1\, 9\,3, 5\, 4)||(6\,7\,8).$$

\end{example}
\begin{theorem}\label{inversecombina}Let $S$ be a $^+1$-monoid. Then, the M\"obius function of the posets $\Pi_{S,r}(n)$ and $\mathbf{P}_{S,r}(n)$, $n\in\N$, give us respectively the matrices $\T_{S,r}$ and $\U_{S,r}$,
	\begin{eqnarray}\label{eq.invmatrix1}
	\T_{S,r}(n,k)&=&\sum_{(\mathbf{V},\pi)\in \Pi_{S,r}(n,k)}\mu(\widehat{0},(\mathbf{V},\pi)),\\\label{eq.invmatrix2}
	\U_{S,r}(n,k)&=&\sum_{(\boldsymbol{\ell},\sigma)\in \mathbf{P}_{S,r}(n,k)}\mu(\widehat{0},(\boldsymbol{\ell},\sigma)).
	\end{eqnarray}
	
\end{theorem}
\begin{proof}
We begin by defining the M\"obius cardinal of the sets $\Pi_{S,r}(n,k)$ and $\mathbf{P}_{S,r}(n,k)$, \begin{eqnarray*}|\Pi_{S,r}(n,k)|_{\mu}&:=&\sum_{(\mathbf{V},\pi)\in \Pi_{S,r}(n,k)}\mu(\widehat{0},(\mathbf{V},\pi)),\\
|\mathbf{P}_{S,r}(n,k)|_{\mu} &:=&\sum_{(\boldsymbol{\ell},\sigma)\in \mathbf{P}_{S,r}(n,k)}\mu(\widehat{0},(\boldsymbol{\ell},\sigma)).\end{eqnarray*}
In order to prove Equation (\ref{eq.invmatrix1}), since $\M_{S,r}=|\Pi_{S,r}(n,k)|$, it is enough to prove that for every $j \leq  n$,\; $n,j\in \N$,
\begin{equation}\label{eq.duality1}\sum_{k=0 }^n|\Pi_{S,r}(n,k)|_{\mu}\;|\Pi_{S,r}(k,j)|=\delta_{n,j}.
\end{equation}
Similarly, Equation (\ref{eq.invmatrix2}) is equivalent to
\begin{equation}\label{eq.duality2}\sum_{k=0 }^n|\mathbf{P}_{S,r}(n,k)|_{\mu}\;|\mathbf{P}_{S,r}(k,j)|=\delta_{n,j}.
\end{equation}
Let $(\mathbf{V}',\pi')$ be an element of $\Pi_{S,r}(n,j)$. By properties of the M\"obius function we have that
\begin{equation*}
\sum_{\widehat{0}\leq (\mathbf{V},\pi)\leq (\mathbf{V}',\pi')}\mu(\widehat{0},(\mathbf{V},\pi))=\delta(\widehat{0},(\mathbf{V}',\pi'))=\delta_{n,j}.
\end{equation*}
Summing over all the elements of $\Pi_{S,r}(n,j)$, interchanging sums and classifying by the size of $\pi$, we get
\begin{eqnarray*}
\delta_{n,j}&=&\sum_{(\mathbf{V}',\pi')\in\Pi_{S,r}(n,j)}\sum_{\widehat{0}\leq (\mathbf{V},\pi)\leq (\mathbf{V}',\pi')}\mu(\widehat{0},(\mathbf{V},\pi))\\&=&\sum_{k=0}^n\sum_{(\mathbf{V},\pi)\in\Pi_{S,r}(n,k)}\left(\sum_{(\mathbf{V}',\pi')\geq(\mathbf{V},\pi)}\mu(\widehat{0},(\mathbf{V},\pi))\right)\\&=&\sum_{k=0}^n\sum_{(\mathbf{V},\pi)\in\Pi_{S,r}(n,k)}\mu(\widehat{0},(\mathbf{V},\pi))\left(\sum_{(\mathbf{V}',\pi')\geq(\mathbf{V},\pi)}1\right)\\&=&\sum_{k=0}^n\sum_{(\mathbf{V},\pi)\in\Pi_{S,r}(n,k)}\mu(\widehat{0},(\mathbf{V},\pi))|\{(\mathbf{V}',\pi')|(\mathbf{V}',\pi')\geq(\mathbf{V},\pi)\}|.
\end{eqnarray*}
From this, by Equation (\ref{eq.coideal1}), Lemma \ref{coideal}, we obtain Equation (\ref{eq.duality1}).  Equation (\ref{eq.duality2}) can be proven in a similar manner.
\end{proof}

\begin{example}
	Let $S$ be, as in Example \ref{stirling2inverse}, the $^+ 1$-monoid of odd integers. It is not difficult to check that $$\La_{S,r}=\left\langle (1-x^2)^{-r},\ln\left(\frac{1+x}{1-x}\right)^{\frac{1}{2}}\right\rangle.$$ Since $\ln\left(\frac{1+x}{1-x}\right)^{\frac{1}{2}}$ is the hyperbolic arctangent, we have $$\U_{S,r}=\left\langle \cosh^{-2r}(x),\tanh(x)\right\rangle.$$
	For $r=1$ we have
	\begin{align*}
	\U_{S,1}&=\left\langle \cosh^{-2}(x),\tanh(x)\right\rangle\\
	&=\left(
	\begin{array}{ccccccccc}
	1 & 0 & 0 & 0 & 0 & 0 & 0 & 0 & 0 \\
	0 & 1 & 0 & 0 & 0 & 0 & 0 & 0 & 0 \\
	-2 & 0 & 1 & 0 & 0 & 0 & 0 & 0 & 0 \\
	0 & -8 & 0 & 1 & 0 & 0 & 0 & 0 & 0 \\
	16 & 0 & -20 & 0 & 1 & 0 & 0 & 0 & 0 \\
	0 & 136 & 0 & -40 & 0 & 1 & 0 & 0 & 0 \\
	-272 & 0 & 616 & 0 & -70 & 0 & 1 & 0 & 0 \\
	0 & -3968 & 0 & 2016 & 0 & -112 & 0 & 1 & 0 \\
	7936 & 0 & -28160 & 0 & 5376 & 0 & -168 & 0 & 1 \\
	\end{array}
	\right).
	\end{align*}
	The M\"obius functions of the posets in this case have  very interesting combinatorial interpretations. The absolute values of the first row gives us the number of ``Zag" permutations (or tangent numbers), $z_{2n+1}$ (Sequence A000182 in \cite{OEIS}). The second one gives us the number of cyclically (reverse) alternating permutations $c_{2n+1}$ of order $2n+1$ (Sequence A024283 in \cite{OEIS}). By Theorem \ref{inversecombina} we have
	\begin{eqnarray}
	z_{2n+1}&=&|\sum_{(\ell,\emptyset)\in\mathbf{P}_{S,1}(2n,0)}\mu(\widehat{0},(\boldsymbol{\ell},\emptyset))|,\\\label{cyclic}
	c_{2n+1}&=&|\sum_{(\boldsymbol{\ell},c)\in\mathbf{P}_{S,1}(2n-1,1)}\mu(\widehat{0},(\boldsymbol{\ell},(c)))|.
	\end{eqnarray}
The sum in Equation (\ref{cyclic}) is over  linear order-cyclic permutation pairs.
 For example, for $n=2$, the pairs are of two forms:
	  \begin{enumerate}
	  	\item $a_1\, a_2||(a_3)$
	  	\item $\emptyset||(a_1\, a_2\, a_3)$
	  \end{enumerate}
  Both kinds of pairs cover $\widehat{0}=\emptyset||(1)(2)(3)$. There are $6$ elements of type $(1)$ and $2$ of type $(2)$. Hence,
  $$\sum_{(\boldsymbol{\ell},(c))\in\mathbf{P}_{S,1}(3,1)}\mu(\widehat{0},(\boldsymbol{\ell},(c)))=-8.$$

  The number of cyclically (reverse) alternating permutations of size $5$ is $8$,
  \begin{equation*}
  \begin{tabular}{ c c c l }
  2\,4\,3\,5\,1& 3\,4\,2\,5\,1 & 3\,5\,2\,4\,1& 4\,5\,1\,3\,2\\
  2\,5\,3\,4\,1 & 3\,4\,1\,5\,2 & 3\,5\,1\,4\,2&4\,5\,2\,3\,1.
  \end{tabular}
  \end{equation*}

  \end{example}

\section{Some Graph Theoretical Connections}\label{sec5}
\subsection{Restricted Stirling numbers for graphs}
The Stirling numbers for graphs were introduced in \cite{Tomascu} as the number of partitions of $V(G)$ into $k$ independent subsets, i.e., there are no edges between any two vertices included in a subset. Motivated by its strong connection to the chromatic polynomial many authors investigated the properties of these sequences, see \cite{Galvin} for a brief history on these studies. Here we introduce a dual version in order to give a natural interpretation of the $(S,r)-$Stirling number of the second kind, ${n\brace k}_{S,r}$.





Let $G$ be a simple finite graph on $[n]$. We let ${G\brace k}^c$ denote the number of partitions of $V(G)$ into $k$ cliques, i.e.,  such that the induced graph on the vertices of each block is a clique. Let further $B(G)^c$ be the number of partitions of $V(G)$ into cliques. We call ${G\brace k}^c$ the \emph{dual Stirling number of the second kind for graphs} and $B(G)^c$ the \emph{dual Bell-number for graphs}. Clearly, ${G\brace k}^c$ is the dual of ${G \brace k}$ in the sense that ${G\brace k}^c={\overline{G}\brace k}$, where $\overline{G}$ is the complement of the graph $G$. Similarly, $B(G)^c$ is the dual of $B(\overline{G})$, the Bell-number of graphs defined for instance in \cite{Duncan}. Further, given a set of integers $S=\{s_1,s_2,\ldots, s_k\}$, we let ${G\brace k}^c_S$ denote the number of ways to partition $V(G)$ into the union of $k$ occurrences of $K_{s_i}$, where $s_i\in S$ and $K_s$ denotes the complete graph on $s$ vertices. For instance, if $S$ contains only the integer $2$, ${G\brace k}^c_S$ is the number of perfect matchings of the graph $G$. Similarly, we define $B_S(G)^c$ as the number of ways to partition $V(G)$ into cliques of  $K_{s_i}$, where $s_i\in S$. For instance, $B_S(P_n)^c$ with $S=\{1,2\}$, where $P_n$ denotes the path graph on $[n]$, is equal to  the Fibonacci number.

It is clear that if $G$ is the complete graph, ${G\brace k}^c_S$ is the $S$-restricted Stirling number ${n\brace k}_S$. Further, if $G$ is the join of the complete graph on $n$ vertices and the empty graph on $r$ vertices $K_n+E_r$, we have


$$
{K_n+E_r \brace k}^c_{S}={n\brace k}_{S,r}\quad \mbox{and}\quad B_S(K_n+E_r)^c=B_{n,S,r}.
$$
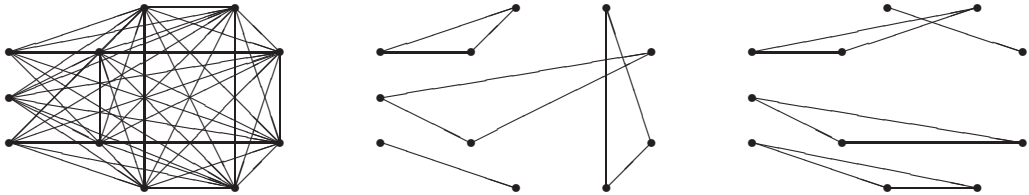
\begin{figure}[ht]
\setlength{\unitlength}{0.6cm}
\begin{picture}(8,6)
\put(1,2){\circle*{0.2}}\put(3,4){\circle*{0.2}}\put(4,1){\circle*{0.2}}
\put(1,3){\circle*{0.2}}\put(4,5){\circle*{0.2}}\put(6,1){\circle*{0.2}}
\put(1,4){\circle*{0.2}}\put(6,5){\circle*{0.2}}\put(7,2){\circle*{0.2}}
\put(3,2){\circle*{0.2}}\put(7,4){\circle*{0.2}}
\put(1,2){\line(3,-1){3}}\put(1,2){\line(1,0){2}}\put(1,2){\line(1,1){2}}\put(1,2){\line(1,1){3}}
\put(1,2){\line(5,3){5}}\put(1,2){\line(3,1){6}}\put(1,2){\line(1,0){6}}\put(1,2){\line(5,-1){5}}
\put(1,3){\line(3,-2){3}}\put(1,3){\line(2,-1){2}}\put(1,3){\line(2,1){2}}\put(1,3){\line(3,2){3}}
\put(1,3){\line(5,2){5}}\put(1,3){\line(6,1){6}}\put(1,3){\line(6,-1){6}}\put(1,3){\line(5,-2){5}}
\put(1,4){\line(1,-1){3}}\put(1,4){\line(1,0){6}}\put(1,4){\line(3,1){3}}\put(1,4){\line(5,1){5}}
\put(1,4){\line(3,-1){6}}\put(1,4){\line(5,-3){5}}
\put(4,1){\line(-1,3){1}}\put(4,1){\line(0,1){4}}\put(4,1){\line(1,2){2}}
\put(4,1){\line(1,1){3}}\put(4,1){\line(3,1){3}}\put(4,1){\line(1,0){2}}
\put(3,2){\line(0,1){2}}\put(3,2){\line(1,3){1}}\put(3,2){\line(1,1){3}}\put(3,2){\line(2,1){4}}
\put(3,2){\line(1,0){4}}\put(3,2){\line(3,-1){3}}
\put(3,4){\line(1,1){1}}\put(3,4){\line(3,1){3}}\put(3,4){\line(1,0){4}}\put(3,4){\line(2,-1){4}}
\put(3,4){\line(1,-1){3}}\put(4,5){\line(1,0){2}}\put(4,5){\line(3,-1){3}}\put(4,5){\line(1,-1){3}}
\put(4,5){\line(1,-2){2}}\put(6,5){\line(1,-1){1}}\put(6,5){\line(1,-3){1}}\put(6,5){\line(0,-1){4}}
\put(7,4){\line(0,-1){2}}\put(7,4){\line(-1,-3){1}}\put(7,2){\line(-1,-1){1}}
\end{picture}
\begin{picture}(8,6)
\put(1,2){\circle*{0.2}}\put(3,4){\circle*{0.2}}\put(4,1){\circle*{0.2}}
\put(1,3){\circle*{0.2}}\put(4,5){\circle*{0.2}}\put(6,1){\circle*{0.2}}
\put(1,4){\circle*{0.2}}\put(6,5){\circle*{0.2}}\put(7,2){\circle*{0.2}}
\put(3,2){\circle*{0.2}}\put(7,4){\circle*{0.2}}
\put(1,2){\line(3,-1){3}}\put(1,3){\line(2,-1){2}}\put(1,3){\line(6,1){6}}\put(3,2){\line(2,1){4}}
\put(1,4){\line(1,0){2}}\put(1,4){\line(3,1){3}}\put(3,4){\line(1,1){1}}
\put(6,1){\line(0,1){4}}\put(7,2){\line(-1,3){1}}\put(6,1){\line(1,1){1}}
\end{picture}
\begin{picture}(8,6)
\put(1,2){\circle*{0.2}}\put(3,4){\circle*{0.2}}\put(4,1){\circle*{0.2}}
\put(1,3){\circle*{0.2}}\put(4,5){\circle*{0.2}}\put(6,1){\circle*{0.2}}
\put(1,4){\circle*{0.2}}\put(6,5){\circle*{0.2}}\put(7,2){\circle*{0.2}}
\put(3,2){\circle*{0.2}}\put(7,4){\circle*{0.2}}
\put(1,2){\line(3,-1){3}}\put(1,2){\line(5,-1){5}}\put(4,1){\line(1,0){2}}
\put(1,3){\line(2,-1){2}}\put(1,3){\line(6,-1){6}}\put(3,2){\line(1,0){4}}
\put(1,4){\line(1,0){2}}\put(1,4){\line(5,1){5}}\put(3,4){\line(3,1){3}}
\put(4,5){\line(3,-1){3}}
\end{picture}
\caption{$K_8+E_3$ and two examples of a partition of $K_8+E_3$ into $k=4$ blocks with $S=\{2,3\}$.}
\end{figure}
\subsection{Acyclic orientations of the complete bipartite graph}

Let $G=(V,E)$ be a simple graph with vertex set $V$, $|V|=n$, and edge set $E$, $|E|=m$. An \emph{acyclic orientation} $\overrightarrow{G}$ of the undirected graph $G$ is an assignment of a direction to each edge of the graph such that there are no directed cycles.
Let $A(G)$ be the number of acyclic orientations of the graph $G$; it is an interesting graph parameter with unexpected connections to the chromatic polynomial of a graph.

Bipartite graphs are crucial in the theory of acyclic orientations, and interestingly $A(K_{n_1,n_2})$ leads to the natural appearance of Stirling numbers.
Let $K_{n_1,n_2}$ be the \emph{complete bipartite graph} on $n=n_1+n_2$ vertices. $K_{n_1,n_2}$ is the graph with vertex set $A\cup B$, where $A=\{u_1,\ldots, u_{n_1}\}$ and $B=\{v_1,\ldots, v_{n_2}\}$, and edge set $E=\{(u,v)| u\in A \mbox{ and } v\in B\}$. $A$ and $B$ are called the \emph{bipartite blocks}. It is known \cite{Cameron} that
$$
A(K_{n_1,n_2})=\Ber_{n_1}^{(-n_2)}=\sum_{m=0}^{\min{n_1,n_2}}(m!)^2{n_1+1\brace m+1}{n_2+1\brace m+1},
$$
where $B_{n_1}^{(-n_2)}$ is the poly-Bernoulli number of negative indices \cite{Kaneko}. (We refer to these numbers in a later section.)

Next, we present an example of a graph such that the number of acyclic orientations is given by a modified poly-Bernoulli number, involving ${n\brace k}_{S,r}$.
The \emph{degree} of a vertex $v$, $\mbox{deg}(v)$, is the number of edges adjacent to $v$. The vertices of the bipartite block $A$ have degree $|B|$ and the vertices of the $B$ all have degree $|A|$. Let $\mbox{deg}_{o}(v)$ denote the \emph{outdegree} of the vertex $v$, the number of edges $e$ whose starting vertex is $v$.

Let $S=\{s_1,\ldots, s_k\}$ be a set of positive integers. Let $A^*=\{a_1,a_2,\ldots, a_r\}$ and $B^*=\{b_1,b_2,\ldots, b_r\}$ be two $r-$sets of vertices. We let $K_{n_1+r,n_2+r}$ denote the complete bipartite graph with bipartite blocks $A\cup A^*$ and $B\cup B^*$.
Further, we let $\widehat{K}_{n_1+r,n_2+r}$ denote the complete bipartite graph on $\widehat{A}=A\cup A^*\cup\{\overline{u}\}$ and $\widehat{B}=B\cup B^*\cup \{\overline{v}\}$.
Consider the acyclic orientations of the $\widehat{K}_{n_1+r,n_2+r}$ with the following properties:
\begin{itemize}
\item[$S$]: $\forall v,w\in A\cup A^*$: $\mbox{deg}_{o}(v)-\mbox{deg}_{o}(w)\in S$, and  analogously for all $ v,w\in B\cup B^*$. This means that the outdegrees of two vertices in $A\cup A^*$ or $B\cup B^*$ differ only by a number contained in $S$.
\item[$r$]: if $u,v\in A^*$ then $\mbox{deg}_{o}(u)\not = \mbox{deg}_{o}(v)$, and analogously for $u,v\in B^*$.
\item[$\overline{ss}$]: $\overline{u}$ is the unique source (vertices without ingoing edges) and $\overline{v}$ is the unique sink (vertex without outgoing edges).
\end{itemize}
Let $A_{(r,S,\overline{ss})}(\widehat{K}_{n_1+r,n_2+r})$ denote the number of acyclic orientations of $\widehat{K}_{n_1+r,n_2+r}$ satisfying the conditions given above.
Condition $S$ could also be formulated the following way: the number of vertices with the same outdegree in a bipartite block $A\cup A^*$ resp.~ $B\cup B^*$ is contained in $S$. In Figure 2 we give an example with $n_1=n_2=4$, $r=2$, and $S=\{2,3,4\}$, which is associated with the sequence $\color{blue}{\{2,4\}}\color{red}{\{2,4,5\}}\color{blue}{\{1,5,3,6\}}\color{red}{\{1,3,6\}}$. We only draw the edges of $\widehat{K}_{4+2,4+2}$ that are oriented from the set $\widehat{A}$ to $\widehat{B}$. The edges that are not drawn are oriented from the set $\widehat{B}$ to $\widehat{A}$.
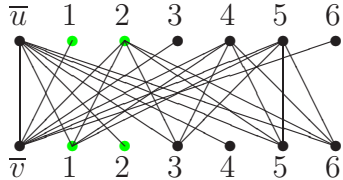
\begin{figure}[ht]
\setlength{\unitlength}{0.7cm}
\begin{picture}(9,4)
\put(0.8,3.3){$\overline{u}$}\put(1.8,3.3){$1$}\put(2.8,3.3){$2$}
\put(3.8,3.3){$3$}\put(4.8,3.3){$4$}\put(5.8,3.3){$5$}\put(6.8,3.3){$6$}
\put(0.8,0.4){$\overline{v}$}\put(1.8,0.4){$1$}\put(2.8,0.4){$2$}
\put(3.8,0.4){$3$}\put(4.8,0.4){$4$}\put(5.8,0.4){$5$}\put(6.8,0.4){$6$}
\put(1,3){\circle*{0.2}}\put(2,3){\textcolor{green}{\circle*{0.2}}}\put(3,3){\textcolor{green}{\circle*{0.2}}}\put(4,3){\circle*{0.2}}
\put(5,3){\circle*{0.2}}\put(6,3){\circle*{0.2}}\put(7,3){\circle*{0.2}}
\put(2,1){\textcolor{green}{\circle*{0.2}}}\put(3,1){\textcolor{green}{\circle*{0.2}}}\put(4,1){\circle*{0.2}}\put(5,1){\circle*{0.2}}
\put(6,1){\circle*{0.2}}\put(7,1){\circle*{0.2}}\put(1,1){\circle*{0.2}}
\put(1,3){\line(1,-2){1}}\put(1,3){\line(1,-1){2}}\put(1,3){\line(3,-2){3}}\put(1,3){\line(2,-1){4}}\put(1,3){\line(5,-2){5}}
\put(1,3){\line(3,-1){6}}\put(1,3){\line(0,-1){2}}
\put(3,3){\line(-1,-1){2}}\put(3,3){\line(-1,-2){1}}\put(3,3){\line(1,-2){1}}\put(3,3){\line(3,-2){3}}\put(3,3){\line(2,-1){4}}
\put(6,3){\line(-5,-2){5}}\put(6,3){\line(-2,-1){4}}\put(6,3){\line(-1,-1){2}}\put(6,3){\line(0,-1){2}}\put(6,3){\line(1,-2){1}}
\put(5,3){\line(-2,-1){4}}\put(5,3){\line(-3,-2){3}}\put(5,3){\line(-1,-2){1}}\put(5,3){\line(1,-2){1}}\put(5,3){\line(1,-1){2}}
\put(2,3){\line(-1,-2){1}}
\put(4,3){\line(-3,-2){3}}
\put(7,3){\line(-3,-1){6}}
\end{picture}
\caption{An acyclic orientation of $\widehat{K}_{4+2,4+2}$.}
\end{figure}
\begin{theorem}
We have
\begin{align}\label{ao}
A_{(r,S,\overline{ss})}(\widehat{K}_{n_1+r,n_2+r})=\sum_{k=0}^{\min(n_1,n_2)} (k+r)!^2 {n_1\brace k}_{S,r} {n_2\brace k}_{S,r}.
\end{align}
\end{theorem}
\begin{proof}
We follow the proof of \cite{Cameron} and apply it to this particular case. Colour the vertices of $A$ red and those of $B$ blue. Any acyclic orientation of the complete bipartite graph $\widehat{K}_{n_1+r,n_2+r}$ can be obtained by ordering the vertices of the graph and orienting each edge from the smaller to larger index. In this arrangements red and blue sequences alternate. The order of vertices inside a sequence of the same colour is irrelevant, since there are no edges between those vertices. Hence, an acyclic orientation can be determined by an alternating sequence of red and blue blocks of the vertices of $\widehat{K}_{n_1+r,n_2+r}$. We now consider the conditions in turn. Condition [$S$] gives bounds on the size of the blocks of the same colour. Condition [$r$] forbids having two vertices from $A^*$ (resp.~$B^*$) in the same block. Condition [$\overline{ss}$] means that the alternating sequence starts with a red block containing the single element $\overline{u}$ and ends with the blue block containing the single vertex $\overline{v}$. Fix $k$, the number of the non-special blocks (blocks that do not contain any elements of $A^*$ resp.~ $B^*$).
We obtain the alternating sequence of the vertices of $\widehat{K}_{n_1+r,n_2+r}$ by determining an ordered partition of the $(n_1+r)$ red elements into $(k+r)$ blocks and an ordered partition of the $(n_2+r)$ blue elements into $(k+r)$ blocks satisfying the given special conditions. This can be done in $(k+r)!^2 {n_1\brace k}_{S,r} {n_2\brace k}_{S,r}$ ways. Summing over $k$ we obtain the theorem.
\end{proof}
\begin{remark}
There are several classically studied objects that count alternating sequence of blocks, such as lonesum matrices, Callan permutations, Vesztergombi permutations (permutations with a bound on the distance between every element and its image), permutations of $[n+k]$
with excedance set $[k]$, and so on (cf. \cite{BB,BH1,BH2,Brew,Cameron}). For instance, there is a natural bijection between lonesum matrices and acyclic orientations of complete bipartite graphs \cite{Cameron}. In every interpretation we can formulate conditions which correspond to the restrictions given by $r$ and the index set $S$. This theorem could be formulated for many other combinatorial objects using these well-studied bijections.
\end{remark}

\section{Some Applications in Special Polynomials}
The Stirling numbers of first and second kind arise in the closed expressions of poly-Bernoulli and poly-Cauchy numbers, number arrays that received a lot of attention recently in number theory and combinatrics. Here, we introduce a generalization of the poly-Bernoulli and poly-Cauchy numbers, the $(S,r)$-poly-Bernoulli, resp. $(S,r)$-poly-Cauchy numbers. The \emph{poly-Bernoulli numbers}  $\Ber_n^{(\mu)}$ were introduced by Kaneko \cite{Kaneko} using the exponential generating function
\begin{align*}
\frac{\li_\mu(1-e^{-t})}{1-e^{-t}}=\sum_{n=0}^{\infty}\Ber_n^{(\mu)}\frac{t^n}{n!}, \quad \mu \in \Z,
\end{align*}
where
\begin{align*}
\li_\mu(t)=\sum_{n=1}^\infty\frac{t^n}{n^\mu}
\end{align*}
is the $\mu$-th \emph{polylogarithm function}. If $\mu=1$ we get $\Ber_n^{(1)}=(-1)^n\mathbf{B}_n$ for $n\geq 0$, where $\mathbf{B}_n$ are the Bernoulli numbers.

Kaneko  \cite[Theorem 1]{Kaneko} found the following explicit formula for poly-Bernoulli numbers:
\begin{align}
\Ber_n^{(\mu)}=\sum_{k=0}^n{n \brace k}\frac{(-1)^{n-k}k!}{(k+1)^\mu}\,. \label{polber}
\end{align}

The poly-Bernoulli numbers have numerous applications in number theory. In particular, Arakawa and Kaneko \cite{AraKan} showed that the poly-Bernoulli numbers can be expressed as special values at negative arguments of  certain combinations of the generalized zeta function
$$\zeta(k_1, \dots, k_{n-1};s)=\sum_{0<m_1<m_2<\cdots < m_n} \frac{1}{m_1^{k_1}\cdots m_{n-1}^{k_{n-1}}m_n^s}.$$

As we mentioned in a previous section, in combinatorics the poly-Bernoulli numbers $\Ber_n^{(-k)}$ enumerate many objects.

\subsection{$(S,r)$-poly-Bernoulli numbers}
A natural generalization of Equation \eqref{polber} is by means of the $(S,r)$-Stirling numbers of the second kind.  In particular, we define the  \emph{$(S,r)$-poly-Bernoulli numbers}  by the expression:
\begin{align}\label{defipolyB}
\Ber_{n,S,r}^{(\mu)}=\sum_{k=0}^n{n \brace k}_{S, r}\frac{(-1)^{n-k}k!}{(k+1)^\mu}.
\end{align}
For convenience, put
$$
E_S(t)=\sum_{s\in S} \frac{t^{s}}{s!}=\sum_{i\geq 1} \frac{t^{k_i}}{k_i!}.
$$
Notice that  $E_{\Z^+}(t)=e^t-1$.

We can now give the generating function of $(S,r)$-poly-Bernoulli numbers  in terms of $E_S(t)$.
\begin{theorem}\label{teo1polyB}
The exponential generating function of $(S,r)$-poly-Bernoulli numbers is
$$
\sum_{n=0}^\infty  \Ber_{n,S,r}^{(\mu)}\frac{t^n}{n!}=\left(E_{S-\vec{1}}(-t)\right)^r \frac{{\rm Li}_\mu\bigl(-E_S(-t) \bigr)}{-E_S(-t)}.
$$
\end{theorem}
\begin{proof}
By definition \eqref{defipolyB} and using (\ref{def:rrs2}), we have
\begin{align*}
\sum_{n=0}^\infty  \Ber_{n,S,r}^{(\mu)}\frac{t^n}{n!}&=\sum_{n=0}^\infty\sum_{k=0}^n{n \brace k}_{S, r}\frac{(-1)^{n-k}k!}{(k+1)^\mu}\frac{t^n}{n!}=\sum_{k=0}^\infty\frac{(-1)^k k!}{(k+1)^\mu}\sum_{n=k}^\infty{n \brace k}_{S, r}\frac{(-t)^n}{n!}\\
&=\sum_{k=0}^\infty\frac{(-1)^k k!}{(k+1)^\mu}\frac{1}{k!}\left(\sum_{i\geq 1} \frac{(-t)^{k_i-1}}{(k_i-1)!}\right)^r\left(\sum_{i\geq 1} \frac{(-t)^{k_i}}{k_i!}\right)^k\\
&=\left(E_{S-\vec{1}}(-t)\right)^r\sum_{k=0}^\infty\frac{(-E_S(-t))^k}{(k+1)^\mu}=\left(E_{S-\vec{1}}(-t)\right)^r \frac{{\rm Li}_\mu\bigl(-E_S(-t) \bigr)}{-E_S(-t)}\,.
\end{align*}
\end{proof}

The following exponential generating functions  follow from some particular cases of $S$. We use the notation
$$
E_m(t)=1+t+\frac{t^2}{2!}+\cdots+\frac{t^m}{m!},
$$
with $E_0=1$, to denote partial sums of the Taylor series for $e^x$. Moreover, let $\E$ and $\Odd$ denote the even and odd positive integers, respectively. We then have the generating functions
\begin{align*}
\sum_{n=0}^\infty \Ber_{n,\Z^+, r}^{(\mu)}\frac{t^n}{n!}&=\frac{e^{-rt}{\rm Li}_\mu\bigl(1-e^{-t}\bigr)}{1-e^{-t}}\,,\\
\sum_{n=0}^\infty \Ber_{n,\leq m,r}^{(\mu)}\frac{t^n}{n!}&=\frac{\bigl(E_{m-1}(-t)\bigr)^r{\rm Li}_\mu\bigl(1-E_{m}(-t)\bigr)}{1-E_{m}(-t)}\,,\\
\sum_{n=0}^\infty \Ber_{n,\geq m,r}^{(\mu)}\frac{t^n}{n!}&=\frac{\bigl(e^{-t}-E_{m-2}(-t)\bigr)^r{\rm Li}_\mu\bigl(E_{m-1}(-t)-e^{-t}\bigr)}{E_{m-1}(-t)-e^{-t}}\,,\\
\sum_{n=0}^\infty \Ber_{n,\E,r}^{(\mu)}\frac{t^n}{n!}&=\frac{\bigl(-\sinh t \bigr)^r{\rm Li}_\mu\bigl(1-\cosh(t)\bigr)}{1-\cosh t}\,,\\
\sum_{n=0}^\infty \Ber_{n,\Odd,r}^{(\mu)}\frac{t^n}{n!}&=\frac{\bigl(\cosh t\bigr)^r{\rm Li}_\mu\bigl(\sinh t \bigr)}{\sinh t}\,.
\end{align*}

The numbers $\Ber_{n,\leq m,r}^{(\mu)}$ and $\Ber_{n,\geq m,r}^{(\mu)}$ are called the \emph{incomplete $r$-poly-Bernoulli number}s, and were studied in detail by Komatsu and Ram\'irez \cite{KJL2}. The particular case $r=0$ was  studied by Komatsu et al. \cite{KLM}.

\subsection{$(S,r)$-poly-Cauchy numbers}
Komatsu \cite{Komc} introduced the \emph{poly-Cauchy numbers of the first kind}, $c_{n}^{(\mu)}$, through the  expression
\begin{align*}
c_{n}^{(\mu)}=\underbrace{\int_{0}^1 \cdots \int_{0}^1}_{\mu}(t_1\cdots t_\mu)_n \,dt_1\cdots dt_\mu.
\end{align*}
Here, $(t)_n$ is the falling factorial defined by $(t)_n=t(t-1)\cdots (t-n+1), n\geq 1,$ and $(t)_0=1$. The exponential generating function of $c_{n}^{(\mu)}$ is
\begin{align*}
\lif_\mu(\ln(1+t))=\sum_{n=0}^{\infty}c_n^{(\mu)}\frac{t^n}{n!}, \quad (\mu\in \Z)
\end{align*}
where
$$\lif_\mu(t)=\sum_{n=0}^\infty\frac{t^n}{n!(n+1)^\mu}$$
is the $\mu$-th \emph{polylogarithm factorial function}. The sequence $c_{n}^{(\mu)}$ is a generalization of the classical \emph{Cauchy numbers} $c_n$.  In particular, with $\mu=1$, we have $c_n^{(1)}=c_n$.   See  \cite{Comtet, Merlini} for general information about Cauchy numbers.

The poly-Cauchy numbers of the first kind  can be defined in terms of Stirling number of the first kind ${n \brack k}$ using the  formula
\begin{align*}
c_n^{(\mu)}=\sum_{k=0}^n{n \brack k}\frac{(-1)^{n-k}}{(k+1)^\mu}\,.
\end{align*}

We define the \emph{$(S,r)$-poly-Cauchy numbers of the first kind} by the expression:
\begin{align}
c_{n,S,r}^{(\mu)}&=\sum_{k=0}^n{n \brack k}_{S, r}\frac{(-1)^{n-k}}{(k+1)^\mu}\,. \label{rc1}
\end{align}

For convenience, put
$$
F_S(t)=\sum_{s\in S}(-1)^{s+1} \frac{t^{s}}{s}=\sum_{i\geq 1}(-1)^{i+1} \frac{t^{k_i}}{k_i},
$$
with $F_{0}=0$. Notice that  $F_{\mathbb{Z}^+}=\ln(1+t)$.

The exponential generating function of the $(S,r)$-poly-Cauchy numbers of the first kind can be given in terms of $F_S(t)$.
\begin{theorem}
The exponential generating function of the $(S,r)$-poly-Cauchy numbers of the first kind is
\begin{align}
\sum_{n=0}^\infty c_{n,S,r}^{(\mu)}\frac{t^n}{n!}=\left(\sum_{s\in S} (-t)^{s-1}\right)^r{\rm Lif}_\mu\bigl(F_S(t)\bigr)\,.
\label{gfrrpcn}
\end{align}
\label{th:gfrpcn}
\end{theorem}
\begin{proof}
From definition \eqref{rc1} and using (\ref{Sgenfunc2}), we have
\begin{align*}
\sum_{n=0}^\infty  c_{n,S,r}^{(\mu)}\frac{t^n}{n!}&=\sum_{n=0}^\infty\sum_{k=0}^n{n \brack k}_{S, r}\frac{(-1)^{n-k}}{(k+1)^\mu}\frac{t^n}{n!}=\sum_{k=0}^\infty\frac{(-1)^k}{(k+1)^\mu}\sum_{n=k}^\infty{n \brack k}_{S, r}\frac{(-t)^n}{n!}\\
&=\sum_{k=0}^\infty\frac{(-1)^k}{(k+1)^\mu}\frac{1}{k!}\left(\sum_{s\in S} (-t)^{s-1}\right)^r\left(\sum_{s \in S} \frac{(-t)^{s}}{s}\right)^k\\
&=\left(\sum_{s\in S} (-t)^{s-1}\right)^r\sum_{k=0}^\infty\frac{1}{k!(k+1)^\mu}\left(\sum_{s \in S} \frac{(-1)^{s+1}t^s}{s}\right)^k\\
&=\left(\sum_{s\in S} (-t)^{s-1}\right)^r{\rm Lif}_\mu\bigl(F_S(t)\bigr)\,.
\end{align*}
\end{proof}

The following exponential generating functions  follow from some particular cases of $S$.
\begin{align*}
\sum_{n=0}^\infty c_{n,\Z^+,r}^{(\mu)}\frac{t^n}{n!}&=\frac{1}{(1+t)^r}{\rm Lif}_\mu\bigl(\ln (1+t)\bigr)\,,\\
\sum_{n=0}^\infty c_{n,\leq m,r}^{(\mu)}\frac{t^n}{n!}&=\left(\frac{1-(-t)^m}{1+t}\right)^r{\rm Lif}_\mu\bigl(F_m(t)\bigr)\,,\\
\sum_{n=0}^\infty  c_{n,\geq m,r}^{(\mu)}\frac{t^n}{n!}&=\left(\frac{(-t)^{m-1}}{1+t}\right)^r{\rm Lif}_\mu\bigl(\ln(1+t)-F_{m-1}(t)\bigr)\,,
\end{align*}
where
$$
F_m(t)=t-\frac{t^2}{2}+\cdots-\frac{(-t)^m}{m},
$$
with $F_0=0$.  The numbers $c_{n,\leq m,r}^{(\mu)}$ and $c_{n,\geq m,r}^{(\mu)}$ are called \emph{incomplete Cauchy numbers} \cite{KJL2}. The particular case $r=0$ was  studied in \cite{Kom2016a}. Moreover,  if $r=0$ and $S=\mathbb{Z}^+$ the generating function reduces to the generating function of the poly-Cauchy numbers (\cite[Theorem 2]{Komc}):
$$
{\rm Lif}_\mu\bigl(\ln(1+t)\bigr)=\sum_{n=0}^\infty c_n^{(\mu)}\frac{t^n}{n!}.
$$

The poly-Cauchy numbers of the second kind $\widehat c_n^{(\mu)}$ \cite[Theorem 4]{Komc} can be also defined by means of Stirling numbers of the first kind:
$$
\widehat c_n^{(\mu)}=\sum_{k=0}^n{n \brack k}\frac{(-1)^{n}}{(k+1)^\mu}\,.
$$
When $\mu=1$, $\widehat  c_n= \widehat  c_n^{(1)}$ are the classical Cauchy numbers of the second kind:
$$
\widehat c_n=\sum_{k=0}^n{n \brack k}\frac{(-1)^{n}}{k+1}=\int_{0}^1t(t+1)\cdots (t+n-1)dt\,.
$$
The generating function of the Cauchy numbers of the second kind is
$$\frac{t}{(1+t)\ln (1+t)}=\sum_{n=0}^\infty \widehat c_n \frac{t^n}{n!}.$$

We define the \emph{$(S,r)$-poly-Cauchy numbers of the second kind} by the expression
\begin{align}
\widehat  c_{n,S,r}^{(\mu)}&=\sum_{k=0}^n{n \brack k}_{S, r}\frac{(-1)^{n}}{(k+1)^\mu}\,. \label{rc1bb}
\end{align}

\begin{theorem}
The exponential generating function of $(S,r)$-poly-Cauchy numbers of the second kind  is
\begin{align}
\sum_{n=0}^\infty \widehat c_{n,S,r}^{(\mu)}\frac{t^n}{n!}=\left(\sum_{s\in S} (-t)^{s-1}\right)^r{\rm Lif}_\mu\bigl(-F_S(t)\bigr)\,.
\end{align}
\label{th:gfrpcn2b}
\end{theorem}

The following exponential generating functions  follow from some particular cases of $S$.
\begin{align*}
\sum_{n=0}^\infty c_{n,\Z^+,r}^{(\mu)}\frac{t^n}{n!}&=\frac{1}{(1+t)^r}{\rm Lif}_\mu\bigl(-\ln (1+t)\bigr)\,,\\
\sum_{n=0}^\infty\widehat c_{n,\leq m,r}^{(\mu)}\frac{t^n}{n!}&=\left(\frac{1-(-t)^m}{1+t}\right)^r{\rm Lif}_\mu\bigl(-F_m(t)\bigr) \,,\\
\sum_{n=0}^\infty\widehat c_{n,\geq m,r}^{(\mu)}\frac{t^n}{n!}&=\left(\frac{(-t)^{m-1}}{1+t}\right)^r{\rm Lif}_\mu\bigl(-\ln(1+t)+F_{m-1}(t)\bigr)\,.
\end{align*}
If $r=0$ and $S=\mathbb{Z}^+$ the generating function  reduces to the generating function of the poly-Cauchy numbers (\cite[Theorem 5]{Komc})
$$
{\rm Lif}_\mu\bigl(-\ln(1+t)\bigr)=\sum_{n=0}^\infty\widehat c_n^{(\mu)}\frac{t^n}{n!}.
$$

\section*{Acknowledgments}
The authors would like to thank the anonymous referees for  reading carefully the paper and giving helpful comments and suggestions.  The research of Jos\'e L. Ram\'irez was partially supported by Universidad Nacional de Colombia, Project No. 37805.


\begin{thebibliography}{99}

\bibitem{AraKan} T.~Arakawa and M.~Kaneko. Multiple zeta values, poly-Bernoulli numbers, and related Zeta functions. Nagoya Math. J. \textbf{153}(1999), 189--209.

\bibitem{Barry} P.~Barry. On a family of generalized Pascal triangles  defined by exponential Riordan arrays. J. Integer Seq. Article 07.3.5 \textbf{10}(2007), 21 pages.

\bibitem{BB} B.~B\'enyi,
{\it Advances in Bijective Combinatorics},
PhD thesis, (2014), available at
http://www.math.u-szeged.hu/phd/dreposit/phdtheses/ benyi-beata-d.pdf.

\bibitem{BH1} B.~B\'enyi and P.~Hajnal.
Combinatorics of poly-Bernoulli numbers.
Studia Sci.~Math.~Hungarica
\textbf{52}(4) (2015), 537--558.

\bibitem{BH2} B.~B\'{e}nyi and P.~Hajnal. Combinatorial properties of poly-Bernoulli relatives. Integers \textbf{17} (2017), A31.

\bibitem{BenRam} B.~B\'enyi and J.~L.~Ram\'irez. Some applications of $S$-restricted set partitions, to appear in Period.~Math.~Hungar. https://doi.org/10.1007/s10998-018-0252-1

\bibitem{Brew} C.~Brewbaker. A combinatorial interpretation of the poly-Bernoulli numbers and two Fermat analogues. Integers \textbf{8} (2008), \# A02, 9 pages.

\bibitem{Broder}  A.~Z.~Broder. The $r$-Stirling numbers. Discrete Math. \textbf{49}(1984), 241--259.


\bibitem{Cameron} P.~J~Cameron, C.~A.~Glass, and  R~U.~Schumacher. Acyclic orientations and poly-Bernoulli numbers,  arXiv 1412.3685v2 (2018), 10 pages.

\bibitem{Comtet} L.~Comtet. \emph{Advanced Combinatorics. The Art of Finite and Infinite Expansions.} D. Reidel Publishing Co., Dordrecht,  The Netherlands, 1974.
\bibitem{Duncan} B.~Duncan and  R.~Peele. Bell and Stirling numbers for graphs. J. Integer Seq. Article 09.7.1, \textbf{12}(2009), 13 pages.

\bibitem{Eng}
J.~Engbers, D.~Galvin, and C.~Smyth, Restricted Stirling and Lah number matrices and their inverses. J. Combin. Theory Ser. A \textbf{161}(2019), 271-298.

\bibitem{Galvin} D.~Galvin and D.~T.~Thank. Striling numbers of forests and cycles, Electron. J. Combin. \textbf{20} (2013), P. 73.

\bibitem{Kaneko}
M.~Kaneko.  Poly-Bernoulli numbers. J. Th\'eor. Nombres Bordeaux \textbf{9}(1997), 199--206.

\bibitem{Ker}
Zs. Keresk\'enyi-Balogh, G. Nyul. Stirling numbers of the second kind and Bell numbers for graphs. Australas. J. Combin. \textbf{58}(2) (2014), 264--274.

\bibitem{Komc}
T.~Komatsu.  Poly-Cauchy numbers.  Kyushu J. Math. \textbf{67}(2013), 143--153.



\bibitem{Kom2016a}
T.~Komatsu.  Incomplete poly-Cauchy numbers. Monatsh. Math. \textbf{180}(2016), 271--288.

\bibitem{KLM}
T.~Komatsu, K.~Liptai, and I.~Mez\H{o}. Incomplete poly-Bernoulli numbers associated with incomplete Stirling numbers. Publ. Math. Debrecen \textbf{88}(2016), 357--368.


\bibitem{KJL2}
T.~Komatsu and  J.~L.~Ram\'irez. Generalized poly-Cauchy and poly-Bernoulli numbers by using incomplete $r$-Stirling numbers. Aequat. Math. \textbf{91}(2017), 1055--1071.

\bibitem{mansour}
T.~Mansour. \emph{Combinatorics of set partitions}. CRC Press, 2012.


\bibitem{Merlini}
D.~Merlini, R.~Sprugnoli, and M.~C.~Verri. The Cauchy numbers. Discrete Math. \textbf{306}(2006), 1906--1920.

\bibitem{Mezo}
I.~Mez\H{o}. The $r$-Bell numbers. J. Integer Seq. Article 11.1.1,  \textbf{14}(2011), 14 pages.



\bibitem{MR2017}
M.~Mihoubi and M.~Rahmani. The partial $r$-Bell polynomials. Afr. Math. \textbf{28}(2017),  1167--1183.


\bibitem{Moll}
V.~Moll, J.~L.~Ram\'irez, and  D.~Villamizar.  Combinatorial and arithmetical properties of the restricted and associated Bell and factorial numbers. To appear in Journal of Combinatorics. arXiv:1707.08138

\bibitem{OEIS}
N.~J.~A.~Sloane. The On-Line Encyclopedia of Integer Sequences. http://oeis.org.

\bibitem{Riordan} L.~W.~Shapiro, S.~Getu, W.~Woan, and L.~Woodson. The Riordan group, Discrete Appl. Math. \textbf{34}(1991), 229--239.


\bibitem{Stanley} R.~P.~Stanley. Acyclic orientations of graphs. Discrete Math. \textbf{5}(1973), 171--178.

\bibitem{Tomascu} I.~Tomascu. Methods combinatoires dans le theorie des automates finnis, Ph.D. Thesis, Bucarest, (1971).

\bibitem{Wakhare} T.~Wakhare. Refinements of the Bell and Stirling numbers. To appear in Transactions on Combinatorics.  https://doi.org/10.22108/TOC.2018.110171.1560

\end{thebibliography}
\end{document}